\documentclass[12pt,toc,page,title]{article}

\usepackage{PRIMEarxiv}

\usepackage[utf8]{inputenc} 
\usepackage{hyperref}       
\usepackage{url}            
\usepackage{booktabs}       
\usepackage{amsfonts}       
\usepackage{nicefrac}       
\usepackage{microtype}      
\usepackage{lipsum}
\usepackage{graphicx}       
\graphicspath{{media/}}     

\usepackage{graphicx}
\usepackage{subcaption}
\usepackage{mathptmx}      
\usepackage[english]{babel}
\usepackage{appendix}
\usepackage{enumitem}
\usepackage{url}            
\usepackage{booktabs}       
\usepackage{amsfonts}       
\usepackage{nicefrac}       
\usepackage{microtype}      
\usepackage{lipsum}
\usepackage{amsthm}
\usepackage{amssymb}
\usepackage{amsmath}
\usepackage{mathptmx}
\usepackage{hyperref}       
\usepackage{xcolor}
\usepackage{comment}
\usepackage{soul}
\usepackage{comment}

\usepackage{algorithm}
\usepackage[noend]{algpseudocode}
\makeatletter
\def\BState{\State\hskip-\ALG@thistlm}
\makeatother

\usepackage[english]{babel}
\DeclareMathOperator*{\esssup}{ess\,sup}

\DeclareMathOperator*{\argmin}{arg\,min}
\DeclareMathOperator*{\argmax}{arg\,max}

\newcommand{\1}{\mathbf{1}}

\newcommand{\ie}{\emph{i.e.}}

 \newtheorem{theorem}{Theorem}
 \newtheorem{lemma}{Lemma}
 \newtheorem{corollary}{Corollary}
 \newtheorem{proposition}{Proposition}
 
 \newtheorem{definition}{Definition}

\title{Games with Incomplete Information Played by Risk-Revising Players
\thanks{This version: March 2026.} 
}

\author{
  Shutian Liu \\
  Department of Systems Engineering \\
  City University of Hong Kong \\
 Hong Kong, China\\
  \texttt{shutian.liu@cityu.edu.hk} \\
}

\begin{document}
\maketitle

\begin{abstract}
This paper introduces risk-revising players to a class of games with incomplete information. These players enter the game with ex ante risk preferences represented by coherent risk measures and develop time-consistent interim revisions of them contingent on their private information. 
The standard Nash equilibrium at ex ante stage and Bayesian Nash equilibrium at interim stage are extended to their risk-averse counterparts. Risk-revising Bayesian Nash equilibrium is proposed to capture behavioral outcomes resulting from interim plays based on revised risk preferences.
We discuss existence results of these equilibrium concepts. 
When players' risk revisions correspond to their ex ante equilibrium play, connections are established between equilibria at the ex ante and interim stages. 
The effect of risk-aversion is analyzed using comparative statics. 
With the help of the dual representation of risk measures, we illustrate the role of risk-aversion in representing inconsistent beliefs.
A numerical example is presented to illustrate the proposed equilibrium concepts under a specific choice of risk preference.
\end{abstract}


\keywords{Risk-revising players \and Games with incomplete information \and Bayesian Nash equilibrium \and Time-consistency}

\section{Introduction}
\label{sec:intro}

Games with incomplete information model strategic interactions where players possess private information about their types. Assuming the existence of a commonly known prior joint probability distribution of players' private types, Bayesian players \cite{harsanyi1967games} update using Bayes rule their beliefs about others' types upon observing their own types drawn from the prior distribution.
In a Bayesian Nash equilibrium (BNE), no Bayesian player can benefit from unilateral deviation given belief.
Due to their capabilities of modeling strategic interactions in the real world  applications in economics, social science, engineering, and computer science, games with incomplete information have been investigated in a range of variants of the standard setting of \cite{harsanyi1967games} with the existence of BNE being established under distinct assumptions on action and type spaces \cite{milgrom1985distributional,meirowitz2003existence,athey2001single,reny2011existence}.

In Harsanyi’s formulation of games with incomplete information, players are assumed to be rational and they are risk-neutral with respect to the uncertainty generated by incomplete information. Accordingly, they evaluate available actions by maximizing the expected value of payoffs, where expectations are taken with respect to their subjective beliefs.
Given a common prior over the space of player types, these beliefs are formed in a consistent manner: each player’s beliefs about others’ types coincide with the conditional distributions induced by the common prior given their own type.
When there is lack of complete information, decision-makers’ (DMs') perceptions of and responds to uncertainty can play a crucial role in shaping observed behavior. 
Ordinal preferences over profiles of strategies may subject to change based on individual risk attitude. 
While the maximization of expected utility provides a canonical model of rational behavior, a spectrum of alternative methods have been proposed in the literature to capture risk-sensitivity, especially within decision-making with a single agent \cite{pratt1978risk,schmeidler1989subjective,artzner1999coherent,tversky1992advances}. 
These criteria have been broadly adopted in decision-making problems both in static and dynamic formulations.
However, formulations of strategic interactions that move beyond the expected utility framework have only recently begun to receive systematic attention \cite{lo1999extensive,kajii2005incomplete,aghassi2006robust,azrieli2011uncertainty,pang2017two,liu2018distributionally,hori2022two,lei2022stochastic,yekkehkhany2022risk,liu2024bayesian}. 
Closely related to the construction adopted in the present paper are, among the others, Su and Xu \cite{su2025incompleterisk} who adopt coherent risk measure and Liu et al. \cite{liu2024bayesian} who use its dual representation, namely a distributionally robust criterion, to represent players' risk preferences about incomplete information. 
Generalization to risk-averse settings is also discussed in Tao and Xu \cite{tao2025generalized}.

Intrinsic to games with incomplete information are three natural stages of decision-making: ex ante, interim, and ex post. 
These three decision stages characterize a progressive refinement of players’ knowledge about the incomplete information, with primary attention devoted to behavioral patterns observable from plays based on interim information.
While this dynamic structure of games with incomplete information captures the dependence of decisions on information revelation, additional challenges can arise when players are not risk-neutral. 
It is well known that choices made by risk‑neutral DMs are time‑consistent, in the sense that plans that are optimal prior to the resolution of uncertainty remain optimal after uncertainty is realized. 
By contrast, risk‑sensitive DMs may face situations in which preferences are not aligned as additional information is observed. 
In such scenarios, a DM may choose to remain ignorant of this inconsistency and adhere to strategies derived from earlier stages despite their subsequent misalignment with realized information.
Alternatively, the decision-maker can choose to acknowledge this inconsistency and respond by adopting revised ordinal preferences or updated decision rules.
Time-consistency issues in single-agent decision-making problems are much more well-investigated than in strategic interactions.
We refer the readers to \cite{boda2006time,bade2022dynamic,klibanoff2007updating,pflug2016time,ruszczynski2006conditional,ruszczynski2010risk,detlefsen2005conditional} and the references therein for a deeper dive into the subject.
In standard Bayesian games in which players' beliefs are consistently derived from a common prior, Nash equilibria formulated ex ante using the joint probability distribution and Bayesian Nash equilibria defined using conditional probability distributions are equivalent \cite{harsanyi1967games}.
This result is used to justify the ex ante formulation of games with incomplete information.
However, similar relationships linking ex ante and interim equilibrium notions may fail to hold in general once risk sensitivity is introduced, due to the issues of time‑inconsistency discussed above.

In this work, we introduce into a class of games with incomplete information a new family of players, referred to as risk-revising players. 
The distinctive feature of risk-revising players lies in that the risk preferences they hold when entering the  ex ante decision stage may change in response to the interim observations of private information. 
Risk‑revising players’ incentives to revise their preferences stem from concerns about potentially time‑inconsistent risk evaluations before and after private information disclosure.
Thus, the family of risk‑revising players contains standard Bayesian players as a degenerate subclass, since Bayesian players are risk-neutral and they have no incentive to revise their risk preferences on the basis of time‑consistency considerations.
Risk-revising players also subsume a class of risk-averse players by allowing adaptation of the level of risk-sensitivity. 
Risk-revision arises only under incomplete information, as full certainty about all of the game parameters eliminates both risk and the need for risk‑revising behavior.

To formally introduce risk-revising players and study their behaviors in strategic interactions, we first extend the Harsanyi formulation of games with incomplete information by introducing risk preference heterogeneity.
In particular, we adopt the setting of finitely many players and endow each of them a distinct ex ante risk preference represented by a law-invariant coherent risk measure \cite{artzner1999coherent,shapiro2013kusuoka}. 
The assigned risk preferences describe players’ risk attitudes at the time they enter the game with incomplete information, that is, prior to any revelation of private information.
The profile of risk preferences is assumed to be common knowledge among all participants, allowing us to focus on how risk-aversion and risk-revision shape behavioral outcomes rather than on players’ conjectures and beliefs about others’ risk preferences.
While this latter aspect is of independent interest, we do not pursue it in the present study.
Then, we construct the risk‑revising feature using the extended conditional risk functionals recently introduced in \cite{pflug2016time}. This approach requires that ex ante risk preferences do not carry over directly to the interim stage. Instead, the conditional counterpart of an ex ante risk preference is endogenously determined as a function of the original preference, the information revealed, and the optimal dual density of the ex ante risk measure associated with the risk induced by a given ex ante strategy profile. 
Accordingly, different types of a given risk‑revising player may arrive at distinct interim revised risk preferences, despite sharing the same ex ante risk preference.
Because risk revision depends on the ex ante strategy profile, the revision procedure is generally non‑unique. Nevertheless, for strategically meaningful ex ante strategy profiles, their induced revised risk preferences can be more preferable according to natural evaluation criteria at the interim stage.

To accommodate risk-preference heterogeneity and risk-revisions, we investigate three classes of equilibrium notions, namely risk-averse Nash equilibrium (RANE), risk-averse Bayesian Nash equilibrium (RABNE), and risk-revised Bayesian Nash equilibrium (RRBNE).
RANE is an equilibrium notion that extends Nash equilibrium in games with incomplete information to players whose preference are represented by coherent risk measures. 
Defined based on ex ante information, plays in a RANE are induced by the quantification of the incomplete information environment using only the prior probability distribution of types. 
When a common prior is assumed, players' subjective perspectives toward the decision environment coincide.
Access to private information is not granted to any player at the time decisions are made in a RANE.
In contrast, RABNE and RRBNE are defined for the interim decision stage.
RABNE concerns the setting in which the interim risk preferences of all types of all Bayesian players are directly assigned rather than derived from the ex ante preferences.
These interim preferences can be interpreted as defined directly with respect to conditional probability distributions. 
We use RABNE to benchmark behavioral outcomes.
Our central focus will be on RRBNE, in which risk-revising players first derive preference revisions and then choose their interim strategies.  
Equilibrium existence is established for all of the notions in behavioral (mixed) strategies.
While the discussion in this paper primarily focuses on finite type and action spaces, the results extend to existence of equilibrium in pure-strategies in the setting of continuous types and actions.
This extension builds on additional axiomatic properties of the risk measures. Fortunately, coherent risk measures suffice to ensure the desired extension.

After setting up equilibria, we turn to the effects resulting from risk-averse attitude and risk-revising features.
We first focus on intra-game equilibrium relationships.
In particular, we examine the consistency of equilibrium plays across different decision stages within a given game of incomplete information. 
Owing to the possibility of time-inconsistent risk evaluation among players, RANE and RABNE are not equivalent in general.
Our objective is then to characterize conditions that are necessary to establish relationships between RANE and RRBNE.
When risk-revisions are induced by the risk associated with a strategy profile and this profile constitutes a RRBNE at the interim stage, then the same strategy profile is also a RANE at the ex ante stage. 
This relation assures that plays that are strategically preferred conditional on private observations are also justified prior to information revelation. 
Furthermore, it also suggests a feasible approach to performing risk-revision, namely revising based on the risks induced by strategy profiles that constitute ex ante RANE.
In the class of games of incomplete information considered in the majority of this paper, it is proper to assume that risk-revising players can identify the ex ante RANE and hence its associated risks. 
The reason lies in that an ex ante formulation of a game of incomplete information is built on a common prior probability distribution of types that is known to all players.
This common prior not only assures that players can consistently derive their subjective beliefs upon private observation, but also renders ex ante play based on the common prior feasible. 
The decomposition technique for extended conditional risk functionals is exploited to establish this relation.
To examine the relation in the opposite direction, specifically whether ex ante RANE admit corresponding interim equilibria, we introduce a structural property of the game.
It stipulates that if a player type ranks one strategy profile above another when evaluated under the revised risk preference induced by the latter, then this ranking is preserved when the evaluation criterion is based on the revised risk preference induced by the former.
We interpret this structural property as a risk-revision analogue of the single-crossing property of incremental returns \cite{athey2001single,milgrom1994monotone}.
We show that, under this structural property and when risk revisions correspond to a RANE, the RANE coincides with an RRBNE.

We also discuss the broader impacts of risk-aversion on games with incomplete information and associated equilibria.
Leveraging the dual representations of coherent risk measures, we show that equilibria arise in games containing a common prior played by risk-averse players can be interpreted as equilibria played by risk-neutral players, or Bayesian players depending on the stage of play, albeit having possibly distinct subjective perspectives.
This observation further motivates using risk-aversion to interpret inconsistent belief systems, which consists of subjective interim beliefs that cannot be derived as conditional probabilities on the basis of a commonly know prior distribution.
Risk-aversion ``commonizes" inconsistent beliefs by enabling the description of these beliefs in an ex ante formulation of games with incomplete information where a common prior is present via shaping conditional probabilities.

The rest of the paper is organized as follows.
In Section \ref{sec:model}, we first present risk-averse games of incomplete information by extending standard ex ante formulations of games of incomplete information with risk preference heterogeneity. 
Then, we introduce risk-revising players as participants of these games. 
Relevant ex ante and interim equilibrium notions are also defined. 
In Section \ref{sec:NE analysis}, we show existence results of the equilibrium notions previously defined and discuss their relationships.
In Section \ref{sec:effect}, we investigate the effects of risk-aversion and describe the approach to representing inconsistent beliefs.
A numerical example is presented in Section \ref{sec:numerical} to illustrate risk-revising behaviors and corroborate some results developed in earlier sections.
Finally, concluding remarks are presented in Section \ref{sec:conclusion}.

\section{Game model}
\label{sec:model}

\subsection{Risk-averse game with incomplete information}
We assume throughout this paper that DMs minimize losses instead of maximize gains.
This assumption streamlines our discussions around risks and risk-averse behaviors and it is without loss of generality. 

There are two main approaches to formulate a game with incomplete information. In the ex ante formulation approach, a common prior probability distribution over players' types is specified and players beliefs are consistently derived from this common prior using private observations of types.
In the interim formulation approach, however, players subjective beliefs are directly specified, there may or may not exist a probability distribution over types that can serve as the common prior. 
In other words, the system of beliefs may not be consistent in the interim formulation.
In this paper, we adopt the ex ante approach since one of our main objectives is to compare risk-revising players with Bayesian players, or, more specifically, compare beliefs obtained merely from Bayesian updating with those affected by risk-aversion.
Discussions concerning the interim formulations are postponed to Section \ref{sec:effect}.
For a more comprehensive discussion between these two formulation approaches, we refer the readers to \cite{van2010interim} and the references therein.

Consider a finite set of players $\mathcal{I}:=\{1,2,\cdots, I\}$ with $I\geq2$. A typical player is denoted by $i\in\mathcal{I}$.
Each player $i$ has a non-empty finite set of types denoted by $T_i$.
Let $T:=\prod_{i\in\mathcal{I}}T_i$.
The information of player $i$ is represented by a filtration $\mathcal{F}^i:=\{\mathcal{F}, \mathcal{F}_\tau^i\}$ where $\mathcal{F}:=\{T\}$ and $\mathcal{F}_\tau^i:=\{T_1\times \cdots T_{i-1}\times \{t_i\}\times T_{i+1}\times \cdots \times T_I: t_i\in T_i\}$.
The sigma algebra $\mathcal{F}$ denotes the information of player $i$ prior to private observation. Hence the same $\mathcal{F}$ is shared by all players ex ante.
The interim information of player $i$ is denoted by the sub-sigma algebra $\mathcal{F}_\tau^i\subset \mathcal{F}$ to indicate the fact that type $t_i\in T_i$ is privately revealed to the player.
On the space $(T, \mathcal{F})$, we specify the probability measure $P$ as the common prior known by all the players.
We assume that $P$ is fully-supported.
Following notational conventions, we occasionally use $\Delta(S)$ to denote the set of probability measures with finite support on set $S$, defined on a proper sigma-algebra.

Each player has a non-empty finite set of actions denoted by $A_i$.
Let $A:=\prod_{i\in\mathcal{I}}A_i$.
A behavioral strategy of player $i$ is a measurable function $\beta_i: T_i\rightarrow\Delta(A_i)$.
We assume that players have perfect recall, hence a behavioral strategy is equivalent to a mixed strategy. 
In the sequel, we simply refer to $\beta_i$ as a mixed strategy to distinguish from pure strategies considered in later sections.
We use the slightly abused notation $\beta_i(a_i|t_i)$ to denote the probability of player $i$ choosing a particular action $a_i\in A_i$ if she is assigned type $t_i\in T_i$.
Let $\beta:=(\beta_1,\cdots,\beta_I)$ denote a strategy profile.
We use $\beta_{-i}:=(\beta_1,\cdots,\beta_{i-1},\beta_{i+1},\cdots, \beta_I)$ to denote the profile of strategies of all players except player $i$.
For any alternative strategy $\beta'_i$ of player $i$, $(\beta'_i, \beta_{-i})$ denotes the strategy profile that coincides with $\beta$ except that player $i$ plays $\beta'_i$.
Note that, in this notation, when $\beta'_i$ is a pure strategy $a'_i$, we write the resulting strategy profile as $(a'_i, \beta_{-i})$.
Let $\Sigma_i$ denote the set of strategies for player $i$.
We use $\Sigma:=\prod_{i\in\mathcal{I}}\Sigma_i$ and $\Sigma_{-i}:=\prod_{j\neq i}\Sigma_j$ to denote the spaces that correspond to $\beta$ and $\beta_{-i}$, respectively.

Let $l^i: T\times A \rightarrow \mathbb{R}$ denote the bounded measurable loss function of player $i$.
Given a type profile $t\in T$ and a strategy profile $\beta\in \Sigma$, let $L^i: T\times \Sigma\rightarrow \mathbb{R}$ denote the average loss of player $i$ defined by 
\begin{equation}
    L^i(t,\beta):=\sum_{a\in A}l^i(t,a)\prod_{i\in \mathcal{I}}\beta_i(a_i|t_i).
    \label{eq:L^i}
\end{equation}
Note that the ``average" refers to averaging with respect to the product of probabilities specified by the mixed strategies $\beta_i$ chosen independently by players in $\mathcal{I}$.
The average loss $L^i$ is random due to the presence of the type profile $t$.

The above definitions constitute a standard ex ante formulation of games with incomplete information.
When players evaluate the random average losses $L^i$ based on the risk-neutral perspective using expectations with respect to prior probability, standard Nash equilibria at the ex ante stage can be defined.
If one moves on to the interim stage and requires plays to be based on private observations and their induced conditional expectations, one arrives at a Bayesian game \cite{harsanyi1967games} where Bayesian Nash equilibria are commonly adopted. 
We postpone the discussions on forming interim beliefs and the corresponding interim equilibria to Section \ref{sec:equilibrium definitions} and first discuss risk-averse extensions of the ex ante formulation.
Building on the preceding framework, we further introduce an ex ante risk preference profile $\mathcal{R}$ to capture the heterogeneity in risk preferences with respect to the uncertainty induced by types.

To accomplish this construction, we first elaborate on risk functionals.
We assume that players' average losses $L^i$ for $i\in\mathcal{I}$ are contained in the space $\mathcal{L}^{\infty}(T,\mathcal{F},P)$ of all essentially bounded, $\mathbb{R}$-valued random variables on the probability space $(T,\mathcal{F},P)$. 
We refer to $P$ as a reference probability measure. 
A risk functional $\rho$ is a mapping from  $\mathcal{L}^{\infty}(T,\mathcal{F},P)$ to $\mathbb{R}$.
\begin{definition}
\label{def:CRM}
(Coherent risk measure).
A coherent risk measure \cite{artzner1999coherent} is a risk functional $\rho:\mathcal{L}^{\infty}(T,\mathcal{F},P)\rightarrow \mathbb{R}$ that satisfies the following properties:
\\
(i) Monotonicity: $\rho(L_1)\leq \rho(L_2)$ if $L_1\leq L_2$ a.s.; 
\\
(ii) Convexity: $\rho ((1-\alpha)L_1+\alpha L_2)\leq (1-\alpha)\rho(L_1)+\alpha \rho ( L_2)$ for $0\leq \alpha \leq 1$;
\\
(iii) Translation invariance: $\rho(L+c)=\rho(L)+c$ for $c\in\mathbb{R}$;
\\
(iv) Positive homogeneity: $\rho(\alpha L)=\alpha \rho(L)$ for $\alpha>0$.
\end{definition}
Note that in the original definition of coherent risk measures in \cite{artzner1999coherent}, gains are considered instead of losses.
Consequently, a transformation $\varrho: Y\rightarrow \rho(-Y)$ needs to be adopted to recover the original definition.
If we additionally assume that players maximize gains and define average utility by $U^i=-L^i$, then $\varrho$ serves as the appropriate risk functional.
Another class of risk functionals commonly considered is referred to as convex risk measures \cite{follmer2011stochastic}, which satisfy properties (i) to (iii) but not (iv).
While we adopt $\mathcal{L}^{\infty}(T,\mathcal{F},P)$ in this paper, it is also possible to consider other domains for the risk functionals.
We refer to \cite{follmer2011stochastic,ruszczynski2006optimization} for those options.
Another property that will be helpful to us is defined as follows.
\begin{definition}
A risk functional $\rho$ is law-invariant if $\rho(Y_1)=\rho(Y_2)$ whenever the random variables $Y_1$ and $Y_2$ have the same probability distribution, \ie, $P(Y_1\leq y)=P(Y_2\leq y)$ for all $y\in\mathbb{R}$.
\end{definition}

We define the ex ante risk preference profile as $\mathcal{R}:=\{\rho^1,\cdots, \rho^I\}$ where $\rho^i:\mathcal{L}^{\infty}(T,\mathcal{F},P)\rightarrow \mathbb{R}$ for each player $i$ is a law-invariant coherent risk measure capturing the risk preference of the player when she enters the game. 
The assumptions of law invariance and coherence suffice for the results developed in subsequent sections. 
We also identify subsets of properties that are sufficient for particular conclusions.
We present the main object of interest in the following.

\begin{definition}
A risk-averse game with incomplete information $\mathbb{G}$ is defined by a tuple $\mathbb{G}:=\langle \mathcal{I}, A, T, \{l^i\}_{i\in\mathcal{I}}, P, \mathcal{R} \rangle$.
\end{definition}
One can also consider a variation of the game $\mathbb{G}$ in which beliefs are directly assigned to players rather than derived from a common prior.
Then, the corresponding formulation yields an interim representation of a game with incomplete information in which players’ beliefs need not be mutually consistent.

Note that a commonly used decomposition of ex ante games of incomplete information into a basic game and an information structure as in \cite{milgrom1985distributional,gossner2000comparison} is feasible.
In this construction, players' loss functions only depend on a common state of nature in addition to players actions.
Therefore, private types are only used for belief updates.
Although our setup is not explicitly formulated in this way, it naturally admits interpretations involving heterogeneous risk preferences. When the common state of nature is viewed as exogenous randomness, risk preference heterogeneity arises in a straightforward manner.

\subsection{Equilibria played by risk-revising players}
\label{sec:equilibrium definitions}
We consider three equilibrium notions associated with game $\mathbb{G}$.
RANE is defined for ex ante play.
RRBNE is an interim equilibrium notion defined for risk-revising players.
RABNE is defined for interim play of risk-averse Bayesian players, which is used for benchmarking RRBNE.

\paragraph{Ex ante equilibria.}

Given an ex ante risk preference $\rho^i\in \mathcal{R}$, the ex ante risk faced by player $i$ under a strategy profile $\beta$ is given by $\rho^i(L^i(t,\beta))$.
This ex ante risk is evaluated prior to any private information being revealed and therefore reflects only the common prior $P$, before subjective beliefs can be formed from private observations.
The decision-making problem at the ex ante stage of player $i\in\mathcal{I}$ given a strategy $\beta_{-i}\in\Sigma_{-i}$ of other players can be formulated as 
\begin{equation}
    \min_{\beta_i'\in\Sigma_i} \rho^i(L^i((t_i, t_{-i}), (\beta_i'(t_i),\beta_{-i}(t_{-i})))).
    \label{eq:ex ante decision problem of player i}
\end{equation}
In the sequel, we slightly abuse the notations and write the random variable $L^i(t,\beta)$ more compactly as $L_\beta^i(t)$, emphasizing that the stochasticity arises from player types and that the random loss is parameterized by a strategy profile. 
We will also occasionally suppress the explicit dependence on the random types  and write $L_\beta^i$.
Unless stated otherwise, $L_\beta^i$ should always be understood as a random loss for player $i$. 
With this notation, we define the following notion of ex ante equilibrium when players face \eqref{eq:ex ante decision problem of player i} simultaneously.
\begin{definition}
\label{def:RANE}
(Risk-averse Nash equilibrium (RANE)).
A strategy profile $\beta^*=(\beta_1^*,\cdots,\beta_I^*)$ constitutes a RANE of the risk-averse game with incomplete information $\mathbb{G}$ if for each player $i\in\mathcal{I}$, 
\begin{equation*}
    \rho^i(L_{\beta^*}^i)\leq \rho^i(L_{(\beta_i,\beta_{-i}^*)}^i), \text{for all } \beta_i\in \Sigma_i.
\end{equation*}
\end{definition}

While RANE is one of the focuses of this work, there are other modeling choices that emphasize distinct behavioral considerations.
One alternative approach relies on first evaluating the risk of $l^i$ and then consider the expected risk with respect to $\beta$.
This modeling choice departs from the standard formulation of games with incomplete information. In particular, the role of a player’s type becomes largely degenerate in the definition of mixed (behavioral) strategies as mappings from types to actions. 
This is less appealing because such type-contingency is central to the analysis of interim behavior in incomplete information games.
Another alternative modeling approach is to assume that players are fully risk-averse, in the sense that risk aversion applies both to uncertainty over types and to randomization induced by mixed strategies. 
While this specification would require additional structure regarding the definition of reference probability measures, it suggests a promising direction for future research, especially in the context of extensive-form games.

\paragraph{Interim equilibria}
The interim equilibrium concept on which we focus is defined for risk‑revising players.
To model how such players revise their risk preferences upon receiving private information at the interim stage, we employ the class of extended conditional risk functionals introduced in \cite{pflug2016time}.
The following dual formulation of risk measures serves as the basis of the construction of this type of conditional risk functionals.
A coherent risk measure $\rho:\mathcal{L}^{\infty}(T,\mathcal{F},P)\rightarrow \mathbb{R}$ admits the following representation via the Fenchel-Moreau duality theorem \cite{artzner1999coherent,ruszczynski2006optimization,follmer2011stochastic}:
\begin{equation}
    \rho(L)=\sup \{ \mathbb{E}(LZ): Z\in\mathfrak{M}\subset\mathcal{L}^1(T,\mathcal{F},P) \},
    \label{eq:dual representation risk measure}
\end{equation}
where $\mathfrak{M}$ is the dual set consisting of dual densities that are absolutely continuous with respect to the reference probability measure $P$ and satisfy that $Z$ is nonnegative , $Z(T)=1$, and $\mathbb{E}(LZ)\leq \rho(L), \forall L\in \mathcal{L}^{\infty}(T,\mathcal{F},P)$.
For the choice of $\mathcal{L}^{\infty}(T,\mathcal{F},P)$ paired with $\mathcal{L}^1(T,\mathcal{F},P)$, the maximum of \eqref{eq:dual representation risk measure} may not be attained due to the fact that the dual set $\mathfrak{M}$ may not be compact in the weak topology of $\mathcal{L}^{\infty}(T,\mathcal{F},P)$ \cite{shapiro2013kusuoka}.
Despite this potential non-existence of optimal dual variables, we follow the setting adopted in \cite{pflug2016time} to introduce the extended conditional risk functionals.
As mentioned in \cite{pflug2016time}, larger spaces, such as $\mathcal{L}^p(T,\mathcal{F},P)$ with $p\geq1$, can be considered for construction.
With this choice, property (iv) of Definition \ref{def:CRM} leads to the fact that the dual set $\mathfrak{M}$ is a nonempty closed convex subset of $\mathcal{L}^p(T,\mathcal{F},P)$. 
Hence, the maximum of the corresponding dual representation is attained. 
However, these aspects are beyond the scope of the current paper, whose focus is on risk-revising behaviors in games.
For this purpose, assume the existence of the dual variables when it is necessary for discussion.
We refer the readers to \cite{ruszczynski2006optimization,shapiro2013kusuoka} for risk measures defined on different spaces.

Consider a sub-sigma algebra $\mathcal{F}_\tau\subset\mathcal{F}$.
The definition of the extended conditional risk functionals \cite{pflug2016time} is summarized as follows.
\begin{definition}
\label{def:extended conditional risk functionals}
(Extended conditional risk functional).
Let $\rho$ be a law-invariant coherent risk measure. 
For dual variables $Z_\tau$ measurable with respect to $\mathcal{F}_\tau$ that satisfies $Z_\tau\geq 0$, $\mathbb{E}(Z_\tau)=1$, and $\mathbb{E}(LZ_\tau)\leq \rho(L)$ for all $L\in \mathcal{L}^{\infty}(T,\mathcal{F},P)$, the extended conditional risk functional associated with $\rho$ is defined as
\begin{equation}
    \rho_{Z_\tau}(L|\mathcal{F}_\tau):=\esssup\{\mathbb{E}(LZ')| Z'\in \mathfrak{M}_{Z_\tau}\subset \mathcal{L}^1(T,\mathcal{F},P) \},
    \label{eq:extended conditional risk functional definition}
\end{equation}
where $\mathfrak{M}_{Z_\tau}$ denotes the set of dual variables associated with $\rho_{Z_\tau}(L|\mathcal{F}_\tau)$ defined as
\begin{equation}
    \mathfrak{M}_{Z_\tau}:=\{Z':\mathbb{E}(Z'|\mathcal{F}_\tau)=\1,
Z'\geq 0,
\text{and } \mathbb{E}(LZ_\tau Z')\leq \rho(L), \forall L\in\mathcal{L}^{\infty}(T,\mathcal{F},P)  
\}.
\end{equation}
\end{definition}
From Definition \ref{def:extended conditional risk functionals}, we observe that the extended conditional version of a risk measure depends not only on the conditioning information $\mathcal{F}_\tau$ but also on a specific dual variable $Z'$ that is optimal in the sense of \eqref{eq:extended conditional risk functional definition}.
This optimal dual variable, in turn, reflects the dependence of the conditional risk functional on the underlying random loss $L$.
In the following, we will make clear the original risk measure, the conditioning information, and the underlying random loss that serve as the basic elements for defining an extended conditional risk functional.
The specific dual variable is defined by \eqref{eq:extended conditional risk functional definition} when the choice of basic elements are given. 
As a conditional risk function, \eqref{eq:extended conditional risk functional definition} is a mapping from $\mathcal{L}^{\infty}(T,\mathcal{F}, P)\rightarrow\mathcal{L}^{\infty}(T,\mathcal{F}_\tau, P)$.
As $\mathcal{F}_\tau\subset\mathcal{F}$, the reference probability measure $P$ is also a probability measure on $(T,\mathcal{F}_\tau)$.
This implies that $\mathcal{L}^{\infty}(T,\mathcal{F}_\tau, P)\subset\mathcal{L}^{\infty}(T,\mathcal{F}, P) $.

With the above preparations in place, we now return to the game $\mathbb{G}$ and introduce the interim equilibrium notion arising from the strategic interactions of risk‑revising players.
Given an ex ante risk preference profile $\mathcal{R}$, consider the revision $\rho_{Z_{\tau,\beta}^i}^i(\cdot|\mathcal{F}_\tau^i)$ of player $i$'s ex ante preference $\rho^i\in\mathcal{R}$.
The notation $Z_{\tau,\beta}^i$ of the dual variable indicates that it defines the corresponding extended conditional risk functional of player $i$, it is $\mathcal{F}_\tau^i$-measurable, and it is induced by the random loss $L_\beta^i$ when the strategy profile being played is $\beta$.
When a specific type $t_i\in T_i$ of player $i$ is of interest, we write the dual variable as $Z_\beta^i(t_i)$ instead of $Z_{\tau,\beta}^i(t_i)$. 
This notation explicitly shows the conditioning information $t_i$ but drops the dependence on the time index $\tau$ for simplicity.  
The interpretation remains unchanged, since the type $t_i$ is only observed at the interim stage, which naturally determines the timing at which risk is evaluated, as originally captured by $\tau$.
The corresponding interim revised risk of a given random loss $L_{\beta'}^i$ parameterized by a strategy profile $\beta'$ faced by type $t_i$ of player $i$ is $\rho_{Z_{\beta}^i(t_i)}^i(L_{\beta'}^i|t_i)$.
Note that the strategy profile $\beta$ that induces risk preference revision and the strategy profile $\beta'$ that is evaluated using the revised preference need not be identical in general.
Let $Z_{\tau,\beta}:=\{Z_{\tau,\beta}^1,\cdots, Z_{\tau,\beta}^I\}$ denote the concatenation of dual variables that induces players' risk revisions. 
We refer to the interim game in which players have revised their ex ante preferences $\rho^i(\cdot)$ to the interim preferences $\rho_{Z_{\tau,\beta}^i}^i(\cdot|\mathcal{F}_\tau^i)$ as the $Z_{\tau,\beta}$-revised game of $\mathbb{G}$.
In this game, each player is considering the variant of problem \eqref{eq:ex ante decision problem of player i} where risk is evaluated with the interim preference given private observation.
Accordingly, the following states the definition of RRBNE in a $Z_{\tau,\beta}$-revised game.
\begin{definition}
\label{def:RRBNE}
(Risk-revised Bayesian Nash equilibrium (RRBNE)).
A strategy profile $\beta^*=(\beta_1^*,\cdots,\beta_I^*)$ constitutes a $Z_{\tau,\beta}$-revised Bayesian Nash equilibrium in a $Z_{\tau,\beta}$-revised game if for each player $i\in\mathcal{I}$, each type $t_i\in T_i$, and each possible action $a_i\in A_i$, 
\begin{equation*}
    \rho_{Z_{\beta}^i(t_i)}^i(L_{\beta^*}^i|t_i)\leq \rho_{Z_{\beta}^i(t_i)}^i(L_{(a_i,\beta_{-i}^*)}^i|t_i).
\end{equation*}
\end{definition}
Note that we always assume in Definition \ref{def:RRBNE} that  a player' original risk preference be ex ante specified by $\rho^i\in\mathcal{R}$ and her conditioning information is given by $\mathcal{F}_\tau^i$.

We also present the following definition of an interim equilibrium for benchmarking RRBNE.
Let $\rho^i(\cdot|\mathcal{F}_\tau^i): \mathcal{L}^{\infty}(T,\mathcal{F}, P)\rightarrow\mathcal{L}^{\infty}(T,\mathcal{F}_\tau^i, P)$ denote the usual conditional risk functional associated with $\rho^i(\cdot)$.
We can interpret $\rho^i(\cdot|\mathcal{F}_\tau^i)$ as the unrevised interim risk preference.
The relation between $\rho^i(\cdot)$ and $\rho^i(\cdot|\mathcal{F}_\tau^i)$ is analogous to the relation between the expectation $\mathbb{E}(\cdot)$ and the conditional expectation $\mathbb{E}(\cdot |\mathcal{F}_\tau)$.
This unrevised interim risk preference $\rho^i(\cdot|\mathcal{F}_\tau^i)$ can also be obtained by choosing $Z_{\tau,\beta}^i=\1$.
\begin{definition}
\label{def:RABNE}
(Risk-averse Bayesian Nash equilibrium(RABNE)).
A strategy profile $\beta^*=(\beta_1^*,\cdots,\beta_I^*)$ constitutes a RABNE if for each player $i\in\mathcal{I}$, each type $t_i\in T_i$, and each possible action $a_i\in A_i$, 
\begin{equation*}
    \rho^i(L_{\beta^*}^i|t_i)\leq \rho^i(L_{(a_i,\beta_{-i}^*)}^i|t_i).
\end{equation*}
\end{definition}
From Definitions \ref{def:RRBNE} and \ref{def:RABNE}, we observe that both RRBNE and RABNE are interim equilibrium notions that nevertheless remain tied to the ex ante stage. In contrast to RRBNE, RABNE does not rely on a specific loss profile, which in \eqref{eq:extended conditional risk functional definition} is required to define risk revision.

\paragraph{Connection to existing equilibrium notions.}

The above defined equilibrium notions are closely related to existing concepts.
An RANE in Definition \ref{def:RANE} reduces to a standard Nash equilibrium in games with incomplete information if all players are risk-neutral, \ie, $\rho^i(\cdot)=\mathbb{E}(\cdot)$ for all $i\in\mathcal{I}$.
Here, expectation is defined with respect to the prior distribution $P$.
A similar reduction also applies to RRBNE in Definition \ref{def:RRBNE} and RABNE in Definition \ref{def:RABNE} by choosing for each player $i\in\mathcal{I}$ the interim risk preferences, which are both conditional risk functional, to be the conditional expectation risk measure given realization of type $t_i\in T_i$.
Then, RRBNE and RABNE reduce to BNE.
Another way to relate to BNE is to consider Bayesian players who deviate from risk-neutrality.
Specifically, one may define risk functionals directly with respect to players’ beliefs as described in \cite{wu2018bayesian}, rather than defining them with respect to the prior distribution and subsequently considering their conditional versions.

When the risk functionals are chosen as the essential supremum, our setting coincides with the robust games studied in \cite{aghassi2006robust}, under the assumption of a finite type space with a fully supported prior. In this extreme case, distributional information becomes irrelevant, since players’ losses are ultimately determined by types corresponding to the worst‑case scenarios. 
Accordingly, the type space can be viewed as a deterministic uncertainty set, and the specification of a common prior beyond its support is unnecessary.

The two constructions discussed above, namely, the risk‑neutral and the robust approaches, address uncertainty from two extreme perspectives. In the former, players place complete trust in the distributional information encoded in the common prior specified by nature, whereas in the latter, players act based on distribution‑free conjectures about uncertainty.
Although both settings can be accommodated within our framework through particular choices of the ex ante risk‑preference profile, doing so precludes the emergence of risk‑revision behavior. The reason lies in that both the expectation and the essential supremum are time‑consistent risk measures. Indeed, absent special dynamic constructions, they are the only two classes of risk measures that induce time‑consistency \cite{pflug2016time}. 
Consequently, by allowing for a general risk‑preference profile $\mathcal{R}$, our framework enriches the class of games represented by $\mathbb{G}$ by introducing heterogeneity in risk preferences as well as the possibility of risk preference revision.

In the equilibrium notions defined in Definitions \ref{def:RANE} and \ref{def:RRBNE}, players’ risk preferences are assumed to be common knowledge. The recent work by Su and Xu \cite{su2025incompleterisk} introduces a model that captures epistemic uncertainty regarding players’ risk preferences. While incorporating such a perspective would be an interesting direction for future research within our framework, the present paper focuses on risk-revision behavior and therefore adopts commonly known risk preferences.

\section{Equilibrium analysis}
\label{sec:NE analysis}
In this section, we first study the existence of RANE and RRBNE. We then slightly depart from the setting introduced in Section \ref{sec:model} to examine equilibrium existence when the game is defined over continuous type and action spaces. After establishing existence, we analyze the relationships between the equilibrium notions of the game $\mathbb{G}$ defined for different decision stages, leveraging the time‑consistency properties of risk evaluations of a class of risk functionals.

\subsection{Equilibrium existence}
\label{sec:NE existence}
We follow the standard approach by showing the existence of Nash equilibrium using Kakutani's fixed-point theorem \cite{kakutani1941generalization}.
To accomplish this objective, we first present the following result regarding the continuity of risk functional.
\begin{lemma}
\label{lemma:continuity of rho}
(Proposition 3.1 of \cite{ruszczynski2006optimization}).
Suppose that $\rho:\mathcal{L}^{\infty}(T,\mathcal{F},P)\rightarrow \mathbb{R}$ satisfies properties (i) and (ii) in Definition \ref{def:CRM}, then $\rho(\cdot)$ is continuous.
\end{lemma}
As we assume law-invariance and coherence of the ex ante risk functionals for risk-revising players, the above continuous result applies to the game setting we consider.
When risk functionals possess more properties than those required in the above lemma, they enjoy stronger forms of continuity.
For instance, any law-invariant convex risk functional is Lipschitz continuous.

\begin{theorem}
\label{thm:RANE existence}
(Existence of RANE). 
Suppose the ex ante risk preference $\rho^i\in\mathcal{R}$ for each player $i\in\mathcal{I}$ satisfies properties (i) and (ii) in Definition \ref{def:CRM}. 
Then, a RANE exists in the risk-averse game with incomplete information $\mathbb{G}$.
\end{theorem}
\begin{proof}
We proceed by checking the conditions of the best-response correspondence of the game required by Kakutani's fixed-point theorem \cite{kakutani1941generalization}.
Since the strategy space $\Sigma$ contains all probability distributions supported on $T$, it is nonempty, compact, and convex.
Let $BR^i:\Sigma_{-i}\rightrightarrows \Sigma_i$ denote the ex ante best response correspondence of player $i\in\mathcal{I}$ defined as 
\begin{equation*}
    BR^i(\beta_{-i}):=\arg\min_{\beta_i\in\Sigma_i} \rho^i(L_{(\beta_i,\beta_{-i})}^i).
\end{equation*}
Firstly, the best response correspondence $BR^i$ for each $i\in\mathcal{I}$ is nonempty.
This property follows from the continuity of $\rho^i(L_{(\beta_i,\beta_{-i})}^i)$ in $\beta_i$, which builds on the continuity suggested in Lemma \ref{lemma:continuity of rho} and the observation that $L^i_{(\beta_i,\beta_{-i})}$ is linear in $\beta_i$, and the nonemptiness and compactness of $\Sigma_i$.
Under these conditions, Weierstrass's theorem indicates that $BR^i$ is nonempty.
Secondly, $BR^i$ is convex for all $i\in\mathcal{I}$.
To show this, consider $\beta_i', \beta_i''\in BR^i(\beta_{-i})$.
Then, we have
\begin{equation*}
    \rho^i(L_{(\beta_i',\beta_{-i})}^i)\leq \rho^i(L_{(\beta_i,\beta_{-i})}^i), \forall \beta_i\in \Sigma_i,
\end{equation*}
and 
\begin{equation*}
    \rho^i(L_{(\beta_i'',\beta_{-i})}^i)\leq \rho^i(L_{(\beta_i,\beta_{-i})}^i), \forall \beta_i\in \Sigma_i.
\end{equation*}
The preceding relations imply that for all $\alpha\in [0,1]$, 
\begin{equation*}
    \alpha\rho^i(L_{(\beta_i',\beta_{-i})}^i)+(1-\alpha)\rho^i(L_{(\beta_i'',\beta_{-i})}^i)\leq \rho^i(L_{(\beta_i,\beta_{-i})}^i), \forall \beta_i\in\Sigma_i.
\end{equation*}
By property (ii) of Definition \ref{def:CRM} and linearity of $L^i_{(\beta_i,\beta_{-i})}$ in $\beta_i$, we arrive at
\begin{equation*}
    \rho^i(L_{(\alpha\beta_i'+(1-\alpha)\beta_i'',\beta_{-i})}^i)\leq \rho^i(L_{(\beta_i,\beta_{-i})}^i), \forall \beta_i\in\Sigma_i.
\end{equation*}
Hence, $\alpha\beta_i'+(1-\alpha)\beta_i''\in BR^i(\beta_{-i})$ and $BR^i$ is convex.
The best response correspondence $BR:\Sigma\rightrightarrows\Sigma$ defined by $BR:=(BR^1,\cdots, BR^I)$ is then convex-valued.
Thirdly, the continuity of $\rho^i$ also indicates that $BR$ has a closed graph.
Therefore, a RANE exists in $\mathbb{G}$.
\end{proof}
Note that the convexity property (ii) in Definition \ref{def:CRM} can be relaxed to quasi‑convexity without affecting the conclusion of the above theorem, since quasi-convexity suffices to show that the best response correspondences are convex-valued.
Quasi-convexity of the preference functional is also closely related to the definition of uncertainty-aversion of \cite{schmeidler1989subjective}.

Entering the interim stage, we first investigate a general equilibrium existence result centered around RABNE.
Existence of RRBNE then follows due to properties of the extended conditional risk functionals.
A conditional risk functional $\rho(\cdot|\mathcal{F}_\tau)$ is coherent if it satisfies conditional counterparts of the properties (i) to (iv) in Definition \ref{def:CRM}, in which all risk evaluations are conditional on available information represented by $\mathcal{F}_\tau\subset \mathcal{F}$.
For instance, the conditional counterpart of the monotonicity property (i) in Definition \ref{def:CRM} reads $\rho(L_1|\mathcal{F}_\tau)\leq\rho(L_2|\mathcal{F}_\tau)$ if $L_1\leq L_2$ a.s..

\begin{theorem}
\label{thm:RABNE existence}
(Existence of RABNE).
Suppose that the interim risk preference represented by the conditional risk functional $\rho^i(\cdot|\mathcal{F}_\tau^i)$ for each player $i\in\mathcal{I}$ satisfies the conditional counterparts of properties (i) and (ii) in Definition \ref{def:CRM}. 
Then, a RABNE exists in the risk-averse game with incomplete information $\mathbb{G}$.
\end{theorem}
\begin{proof}
The proof proceeds by the same arguments as those used in the proof of Theorem \ref{thm:RANE existence}, adapted to the conditional setting.
\end{proof}
Since the above theorem is stated for all conditional risk functionals which satisfy the conditional counterparts of properties (i) and (ii) in Definition \ref{def:CRM}, the following result immediately leads to the existence of RRBNE in $\mathbb{G}$, which is stated as a corollary of the theorem.
\begin{lemma}
(Theorem 19 of \cite{pflug2016time}).
An extended conditional risk functional $\rho_{Z_\tau}$ defined by \eqref{eq:extended conditional risk functional definition} satisfies all the conditional counterparts of the properties in Definition \ref{def:CRM}. 
\end{lemma}
\begin{corollary}
\label{coro:RRBNE existence}
(Existence of RRBNE).
Suppose that each ex ante risk preference $\rho^i\in\mathcal{R}$ is a law-invariant coherent risk measure and that the revision $\rho_{Z_{\tau,\beta}^i}^i(\cdot|\mathcal{F}_\tau^i)$ is defined by \eqref{eq:extended conditional risk functional definition}.
Then, a $Z_{\tau,\beta}$-revised Bayesian Nash equilibrium exists in the risk-averse game with incomplete information $\mathbb{G}$.
\end{corollary}

\paragraph{Existence in infinite games}
We briefly deviate from the setting in Section \ref{sec:model} and discuss equilibrium existence when the game $\mathbb{G}$ possesses continuous type and action spaces. 
We only present the case of the ex ante decision stage.
The interim counterpart follows a procedure with all risk evaluations performed on a conditional basis.

Consider, for now, the type set $T_i\subset \mathbb{R}^{m_i}$ and the action set $A_i\in \mathbb{R}^{n_i}$ for all $i\in\mathcal{I}$ are compact and convex.
In this setting, we resort to pure (behaviroal) strategies for player $i$ defined as $f_i:T_i\rightarrow\mathbb{R}^{n_i}$ and we assume the set of available strategies is $F_i$. 
We write $f_{-i}(t_{-i})$ for the strategies of players other than $i$ when their types are given by $t_{-i}$.
Accordingly, the ex ante risk reads $\rho^i(l^i(t_1,\cdots,t_I,f_1(t_1),\cdots, f_I(t_I)))$ given a pure strategy profile $f:=(f_1,\cdots, f_I)$.
Then, at an ex ante equilibrium, each player $i\in\mathcal{I}$ solves the following decision-making problem given $f_{-i}$:
\begin{equation*}
   \min_{f_i'\in F_i} \rho^i(l^i((t_i, t_{-i}), (f_i'(t_i), f_{-i}(t_{-i})))).
\end{equation*}
In this scenario, the Debreu-Fan-Glicksberg theorem \cite{debreu1952social,fan1952fixed,glicksberg1952further} is commonly adopted to investigate equilibrium existence. 
The conditions required to establish existence includes, in addition to those required in the setting for mixed strategies in finite games, the (quasi-)convexity of $\rho^i(l^i(t_1,\cdots,t_I,f_1(t_1),\cdots, f_I(t_I)))$ in player $i$'s strategy $f_i$.
This property is instrumental in establishing that the best‑response correspondences are convex‑valued. 
Since our analysis is restricted to pure strategies, the average loss in \eqref{eq:L^i} is not defined, and linearity in strategies is therefore absent in general. As a consequence, (quasi‑)convexity of the payoff function with respect to pure strategies becomes necessary to obtain the desired results. In the presence of ex ante risk preferences, we adopt the following construction to establish the convexity property.

We work with the composite function $\phi^i(\cdot):=\rho^i(\Bar{L}^i(\cdot))$ where we write $l^i(t,a)$ as $[\Bar{L}^i(a)](t)$.
Convexity of $\Bar{L}^i$ in $a_i$ refers to the convexity of $l^i(t,a)$ in $a_i$ for each $t\in T$.
Here, $\Bar{L}^i$ is used to distinguish from the average loss in \eqref{eq:L^i} defined for mixed strategies.
Convexity can be obtained from the following result.
\begin{lemma}
(Lemma 3.1 in \cite{ruszczynski2006optimization}). 
The composite function $\phi^i(\cdot)$ is convex if $\Bar{L}^i(\cdot)$ is convex and $\rho^i$ satisfies properties (i) and (ii) in Definition \ref{def:CRM}.
\end{lemma}
With the above lemma, we can proceed to show the existences of RANE and RABNE in the infinite counterpart of $\mathbb{G}$.

Note that equilibrium uniqueness can also be investigated under stronger structural assumptions on the game under the setting of continuous type and action spaces. These include, for instance, single‑crossing properties and supermodularity conditions. We refer the reader to \cite{su2025existence,su2025incompleterisk} and the references therein for further discussion of equilibrium uniqueness results.

\subsection{Connections between ex ante and interim plays}
In the standard setting of games with incomplete information played by Bayesian players, risk-neutrality of all players leads to the conclusion that ex ante NE and interim BNE are equivalent \cite{harsanyi1967games}.
This property ensures that standard Bayesian games, which are built on the ex ante formulation of games with incomplete information with a predetermined common prior, are able to characterize interim behavior outcomes based on subjective beliefs.
However, this equivalence relation does not hold in general when the assumption of risk-neutrality is relaxed.
We will use a numerical example where players are risk-averse and adopt AV@R measures as their risk preferences to demonstrate this phenomenon in detail in Section \ref{sec:numerical}.
The underlying source of this phenomenon is the potential time-inconsistency between ex ante and interim risk evaluations when players are endowed with general risk preferences. In particular, one random cost may be preferred to another prior to the revelation of private information, while the reverse ranking may emerge once information is observed.
Since a NE and its variants require that no player has an incentive to deviate, inconsistency between a player’s ex ante and interim risk evaluations undermines this requirement across different decision stages in a game with incomplete information and obstructs the investigation of interim plays via ex ante formulations. 
Note that despite the fact that interim formulations where subjective beliefs are directly specified can always be considered, ex ante formulations establish a clear epistemic foundation by allowing the articulation of how subjective beliefs can be derived based on the common knowledge of a prior.
Therefore, ex ante formulations are not less preferred than interim formulations, unless inconsistent beliefs have to exist.

The following decomposition theorem of risk functionals is essential in establishing the relationships of ex ante and interim equilibrium notions.
\begin{lemma}
\label{lemma:decomposition}
(Theorem 21 of \cite{pflug2016time}).
Let $\rho$ denote a law-invariant coherent risk measure. Then, the following holds
\begin{equation}
    \rho(L)=\sup\mathbb{E}[Z_\tau \cdot \rho_{Z_\tau}(L|\mathcal{F}_\tau)],
    \label{eq:decomposition of risk measure}
\end{equation}
where the supremum is among all $\mathcal{F}_\tau$-measurable dual variables $Z_\tau$ satisfying $Z_\tau\geq 0$, $\mathbb{E}(Z_\tau)=1$, and $\mathbb{E}(LZ_\tau)\leq \rho(L)$ for all $L\in \mathcal{L}^{\infty}(T,\mathcal{F},P)$.
\end{lemma}
The above decomposition of the risk functional into its extended conditional counterparts shows that the revised risk preferences coincide with the original unconditional formulation only when a weighted-average of the conditional risk measures is taken under a specific worst‑case scenario. This worst‑case scenario is defined with respect to the dual variables associated with the particular random loss vector that drives the revisions.
Based on this result, we show the following connection between the ex ante and the interim equilibria of $\mathbb{G}$. 
\begin{theorem}
\label{thm:RRBNE is RANE}
Let $\Bar{Z}^i$ denote the optimal dual variable associated with the loss $L_{\beta^*}^i$ and the risk preference $\rho^i$ in the sense of \eqref{eq:dual representation risk measure}, \ie, $\rho^i(L_{\beta^*}^i)=\mathbb{E}(L_{\beta^*}^i\Bar{Z}^i)$.
Suppose that the revised risk preference of player $i\in\mathcal{I}$ is $\rho_{Z_{\tau,\beta^*}^i}(\cdot|\mathcal{F}_\tau^i)$, where $Z_{\tau,\beta^*}^i=\mathbb{E}(\Bar{Z}_{}^i|\mathcal{F}_\tau^i)$.
Then, if $\beta^*$ is a $Z_{\tau,\beta^*}$-revised BNE, it is also a RANE.
\end{theorem}
\begin{proof}
Since $\beta^*$ is a $Z_{\tau,\beta^*}$-revised BNE, the following holds for each $i\in\mathcal{I}$, each $t_i\in T_i$, and each $a_i\in A_i$,
\begin{equation}
    \rho_{Z_{\beta^*}^i(t_i)}^i(L_{\beta^*}^i|t_i)
    \leq 
    \rho_{Z_{\beta^*}^i(t_i)}^i(L_{(a_i,\beta_{-i}^*)}^i|t_i).
    \label{eq:thm3:RRBNE condition}
\end{equation}
Multiplying both sides of inequality \eqref{eq:thm3:RRBNE condition} with the nonnegative dual variable $Z_{\beta^*}^i(t_i)$ and taking expectation with respect to $P$, we obtain for each $i\in\mathcal{I}$, each $t_i\in T_i$, and each $a_i\in A_i$ that
\begin{equation}
    \mathbb{E}\left[Z_{\beta^*}^i(t_i)\cdot  \rho_{Z_{\beta^*}^i(t_i)}^i(L_{\beta^*}^i|t_i)\right]
    \leq
    \mathbb{E}\left[Z_{\beta^*}^i(t_i)\cdot  \rho_{Z_{\beta^*}^i(t_i)}^i(L_{(a_i,\beta_{-i}^*)}^i|t_i)\right].
    \label{eq:thm3:1}
\end{equation}
Since $Z_{\beta^*}^i(t_i)=\mathbb{E}(\Bar{Z}^i|t_i)$ and $\rho^i(L_{\beta^*}^i)=\mathbb{E}(L_{\beta^*}^i\Bar{Z}^i)$, the term on the left-hand side of \eqref{eq:thm3:1} equals $\rho^i(L_{\beta^*}^i)$ due to Lemma \ref{lemma:decomposition}.
Lemma \eqref{lemma:decomposition} also leads to the following inequality concerning the right-hand side of \eqref{eq:thm3:1}:
\begin{equation*}
 \mathbb{E}\left[Z_{\beta^*}^i(t_i)\cdot  \rho_{Z_{\beta^*}^i(t_i)}^i(L_{(a_i,\beta_{-i}^*)}^i|t_i)\right]
 \leq    
 \sup_{Z^i(t_i)} \mathbb{E}\left[Z^i(t_i)\cdot \rho_{Z^i(t_i)}^i(L_{(a_i,\beta_{-i}^i)}^i|t_i) \right]
 =\rho^i(L_{(a_i,\beta_{-i}^i)}^i),
\end{equation*}
where $a_i$ is any pure strategy of player $i$.
As the above inequality holds for any pure strategy, it also holds for any mixed strategy $\beta_i$ of player $i$. This indicates $\rho^i(L_{\beta^*}^i)\leq \rho^i(L_{(a_i,\beta_{-i}^*)}^i)$ for each $i\in\mathcal{I}$ and each $\beta_i\in \Sigma_i$. Therefore,  $\beta^*$ is a RANE.
\end{proof}
From Definition \ref{def:RRBNE}, we observe that any $Z_\tau$ can be used to define a RRBNE as long as it is well-defined according to the conditions required by Definition \ref{def:extended conditional risk functionals}.
Among the infinitely many possibilities for preference revision and its associated interim perspective towards game $\mathbb{G}$, Theorem \ref{thm:RRBNE is RANE} suggests one approach to locate a meaningful revision which consistently bridges ex ante and interim equilibrium plays, provided that the conditions in the theorem are satisfied. 
This approach is adopting the conditional expectations of the dual variable that attains the optimal of the dual formulation of the ex ante risk.
While the computations of the ex ante equilibria and the optimal dual variables might be challenging on their own, assumptions on the information structure of game $\mathbb{G}$ support the construction of this approach.
The reason lies in that the common prior $P$ is always known to all players in $\mathbb{G}$, allowing them to engage in strategic interaction already at the ex ante stage, even though their decisions are ultimately conditioned on private information observed at the interim stage.
Note that under an interim formulation of a game with incomplete information, this construction may no longer be feasible. In that setting, players’ beliefs are specified directly and need not be consistent with a common prior. As a result, ex ante play is not properly defined.

From the proof of Theorem \ref{thm:RRBNE is RANE}, we observe that the ranking of the weighted-average of the revised risks in inequality \eqref{eq:thm3:1} is the essential condition that connects interim and ex ante plays, whereas condition \eqref{eq:thm3:RRBNE condition} serves as its sufficient condition.
One might think that a similar procedure with the help of the weighted-average of the revised risks could lead to relationship between the ex ante and interim plays in the converse direction. 
However, this is in general infeasible.
We will illustrate in the following why $\beta^*$ being a RANE does not necessarily lead to a ranking for reasoning about the candidacy of $\beta^*$ to be a RRBNE in the $Z_{\tau,\beta^*}$-revised game.
Suppose that $\beta^*$ is a RANE, then by definition it holds for all $i\in\mathcal{I}$ that $\rho^i(L_{\beta^*}^i)\leq \rho^i(L_{\hat{\beta}}^i)$ where $\hat{\beta}:=(\hat{\beta}_i,\beta_{-i}^*)$ is the strategy profile constructed by substituting player $i$'s
strategy with any $\hat{\beta}_i$ in $\beta^*$.
Hence, it holds for each player $i$ and her strategy $\hat{\beta}_i$ that
\begin{equation*}
    \mathbb{E}\left[ Z_{\tau,\hat{\beta}}^i\cdot \rho_{Z_{\tau,\hat{\beta}}}^i(L_{\beta^*}^i|\mathcal{F}_\tau^*) \right]
    \leq
    \mathbb{E}\left[ Z_{\tau,\beta^*}^i\cdot \rho_{Z_{\tau,\beta^*}}^i(L_{\beta^*}^i|\mathcal{F}_\tau^*) \right]
    \leq
    \mathbb{E}\left[ Z_{\tau,\hat{\beta}}^i\cdot \rho_{Z_{\tau,\hat{\beta}}}^i(L_{\hat{\beta}}^i|\mathcal{F}_\tau^*) \right],
\end{equation*}
where the first inequality follows from Lemma \ref{lemma:decomposition} and the second inequality is a consequence of $\rho^i(L_{\beta^*}^i)\leq \rho^i(L_{\hat{\beta}}^i)$.
The above inequalities indicate, in particular, that $\beta^*$ being a RANE does not lead to a ranking in the weighted average form between itself and an alternative $\hat{\beta}$ in the $Z_{\tau,\beta^*}$-revised game.
Instead, the inequalities lead to a ranking in the $Z_{\tau, \hat{\beta}}$-revised game.
Furthermore, since the conditions for establishing an equilibrium requires comparing a given strategy profile with all possible unilateral deviations from it of all players, its verification may require analyzing all interim revised games induced by those deviations.
This phenomenon is absent in standard game settings where all players are homogeneously adopting (conditional) expectations as their risk preferences.

The preceding discussion underscores the importance of conducting analyses across different interim revised games. The following result performs such an investigation based on the weighted‑average structure.
\begin{proposition}
\label{prop:ex ante ranking holds if switch}  
Suppose for $i\in\mathcal{I}$ and $\beta^*, \hat{\beta}\in \Sigma$ that 
\begin{equation}
    \mathbb{E}\left[ Z_{\tau,\hat{\beta}}^i \cdot \rho_{Z_{\tau,\hat{\beta}}^i}^i(L_{\beta^*}^i|\mathcal{F}_\tau^i)\right] 
    >
    \mathbb{E}\left[ Z_{\tau,\hat{\beta}}^i \cdot \rho_{Z_{\tau,\hat{\beta}}^i}^i(L_{\hat{\beta}}^i|\mathcal{F}_\tau^i)\right].
    \label{eq:beta* better than beta hat under hat revision}
\end{equation}
Then, the following relation holds
\begin{equation}
    \mathbb{E}\left[ Z_{\tau,\beta^*}^i \cdot \rho_{Z_{\tau,\beta^*}^i}^i(L_{\beta^*}^i|\mathcal{F}_\tau^i)\right]
    >
    \mathbb{E}\left[ Z_{\tau,\beta^*}^i \cdot \rho_{Z_{\tau,\beta^*}^i}^i(L_{\hat{\beta}}^i|\mathcal{F}_\tau^i)\right].
    \label{eq:beta* better than beta hat under * revision}
\end{equation}
\end{proposition}
\begin{proof}
From Lemma \ref{lemma:decomposition}, we observe that the dual variables $Z_{\tau,\beta^*}^i$ is obtained from 
\begin{equation*}
    Z_{\tau,\beta^*}^i\in \argmax_{Z^i} \mathbb{E}\left[Z^i\cdot \rho_{Z^i}^i(L_{\beta^*}^i|\mathcal{F}_\tau^i) \right],
\end{equation*}
where $Z^i\in\mathcal{L}^1(T,\mathcal{F},P)$ satisfies the conditions stated in the lemma.
Then, for all dual variables $Z_{\hat{\beta}}^i$ that also satisfies the conditions in Lemma \ref{lemma:decomposition}, we have
\begin{equation}
    \mathbb{E}\left[Z_{\tau,\beta^*}^i\cdot \rho_{Z_{\tau,\beta^*}^i}^i(L_{\beta^*}^i|\mathcal{F}_\tau^i) \right]
    \geq
    \mathbb{E}\left[Z_{\tau,\hat{\beta}}^i\cdot \rho_{Z_{\tau,\hat{\beta}}^i}^i(L_{\beta^*}^i|\mathcal{F}_\tau^i) \right]
    >
    \mathbb{E}\left[Z_{\tau,\hat{\beta}}^i\cdot \rho_{Z_{\tau,\hat{\beta}}^i}^i(L_{\hat{\beta}}^i|\mathcal{F}_\tau^i) \right]
    \label{eq:lemma:1}
\end{equation}
Since Lemma \ref{lemma:decomposition} also leads to
\begin{equation*}
    Z_{\hat{\beta}}^i\in\argmax_{Z^i}
    \mathbb{E}\left[Z^i\cdot \rho_{Z^i}^i(L_{\hat{\beta}}^i|\mathcal{F}_\tau^i) \right],
\end{equation*}
we obtain by inequality \eqref{eq:lemma:1} that
\begin{equation*}
    \mathbb{E}\left[Z_{\tau,\beta^*}^i\cdot \rho_{Z_{\tau,\beta^*}^i}^i(L_{\beta^*}^i|\mathcal{F}_\tau^i) \right]
    >
    \mathbb{E}\left[Z_{\tau,\hat{\beta}}^i\cdot \rho_{Z_{\tau,\hat{\beta}}^i}^i(L_{\hat{\beta}}^i|\mathcal{F}_\tau^i) \right]
    \geq
    \mathbb{E}\left[Z_{\tau,\beta^*}^i\cdot \rho_{Z_{\tau,\beta^*}^i}^i(L_{\hat{\beta}}^i|\mathcal{F}_\tau^i) \right].
\end{equation*}
This completes the proof. 
\end{proof}
The above proposition states that if the weighted-average of the interim risks corresponding to a given strategy profile exceeds that of an alternative strategy profile when the interim preference revisions and weighting parameters are defined with respect to the alternative profile, then the same ordering also holds when the revisions and weightings are instead defined with respect to the given strategy profile.
This property has the following indication on bridging interim revised games defined with respect to different dual variables.
When interim risks are aggregated in a weighted‑average form, there exists a certain dominance property among the criteria used to define the interim preference revisions and weighting parameters in determining the resulting ranking.

Inequality \eqref{eq:beta* better than beta hat under * revision} also yields in particular the following observation.
Choose $\hat{\beta}=(\hat{\beta}_i, \beta_{-i}^*)$.
Then, $\beta^*$ is not a $Z_{\tau,\beta^*}$-revised BNE, as there is at least one type $t_i\in T_i$ of player $i\in\mathcal{I}$ who can benefit by deviating unilaterally from $\beta^*$ in the $Z_{\tau,\beta^*}$-revised game, \ie, $\rho_{Z_{\tau,\beta^*}^i}^i(L_{\beta^*}^i|t_i)>\rho_{Z_{\tau,\beta^*}^i}^i(L_{\hat{\beta}}^i|t_i)$.
This consequence effectively connects the evaluations of the strategy profile $\beta^*$ in condition \eqref{eq:beta* better than beta hat under hat revision} which concerns the $Z_{\tau,\hat{\beta}}$-revised game with that in the RRBNE of the $Z_{\tau,\beta^*}$-revised game.

Before examining the converse direction of Theorem \ref{thm:RRBNE is RANE}, we note several distinctions between standard games with incomplete information and the settings adopted here, which account for the failure of equivalence between ex ante and interim equilibria.
In standard settings, rearrangement of summations with respect to the common prior joint probability distribution in computing expected losses leads to the equivalence of ex ante expected loss and expected interim (conditional) losses.
This property does not hold in our setting due to Lemma \ref{lemma:decomposition}. In particular, when taking expectations of interim risks, not only is a change of measure introduced, but the worst‑case change of measure must also be computed to equate the ex ante risk.
This worst‑case change of measure restricts us to going only in a single direction in the corresponding inequalities.

To proceed, we introduce a structural property regarding interim revised risks.
\begin{definition}
\label{def:complimentarity in risk revision}    
(Risk-preference revision complementarity). 
A risk functional $\rho$ satisfies the property of risk-preference revision complementarity (RPRC) if for all pairs of strategy profiles $(\beta, \beta')$, $\rho_{Z_{\tau,\beta'}}(L_{\beta}|\mathcal{F}_\tau)
    <
    \rho_{Z_{\tau,\beta'}}(L_{\beta'}|\mathcal{F}_\tau)$ 
    implies 
    $\rho_{Z_{\tau,\beta}}(L_{\beta}|\mathcal{F}_\tau)
    <
    \rho_{Z_{\tau,\beta}}(L_{\beta'}|\mathcal{F}_\tau)$.
\end{definition}

The definition of RPRC requires that if a strategy profile is ranked as less risky than an alternative profile when risks are evaluated using the revised risk functional induced by the alternative profile, then the same ranking  obtains when risks are evaluated using the revised risk functional induced by the original strategy profile.
This property can be interpreted as stating that a player has an incentive to resort to the interim game in which preferences are revised based on a given strategy profile whenever the player finds that profile favorable in other interim games where preferences are revised based on different profiles.
While this structural property is unique to risk-revising players in game $\mathbb{G}$, it can nonetheless be compared to existing notions in game theory.
For instance, the single crossing property of incremental returns in \cite{athey2001single,milgrom1994monotone}, illustrates a relation in a similar spirit.
It states that if a higher action is preferred to a lower action when player has a lower type, then the same relation holds when the player has a higher type.
In contrast to RPRC in our setting, we do not relate types to actions; instead, we state the complementarity between risk revisions and strategy profiles. 
Furthermore, the ordering over types is replaced by time-consistency with respect to ex ante risk preferences.
The single-crossing property of incremental returns is an essential condition for establishing the existence of pure-strategy BNE that are non-decreasing in player types. In contrast, we use RPRC to establish the relationship between ex ante and interim plays in what follows.

\begin{theorem}
\label{thm:RANE is RRBNE}
Let the revised risk preferences be defined as in Theorem \ref{thm:RRBNE is RANE}.
Suppose, in addition, that the common prior $P$ is fully supported, all ex ante risk preference $\rho^i\in\mathcal{R}$ satisfy RPRC, and the dual variable $Z_{\beta}^i(t_i)>0$ for each $i\in\mathcal{I}$, each $t_i\in T_i$, and each $\beta\in\Sigma$.
Then, if $\beta^*$ is a RANE, it is also a $Z_{\tau,\beta^*}$-revised BNE.
\end{theorem}
\begin{proof}
We will proceed by showing the contraposition, that is, if $\beta^*$ is not a $Z_{\tau, \beta^*}$-revised BNE, then it cannot be a RANE.
As $\beta^*$ is not a  $Z_{\tau,\beta^*}$-revised BNE, by definition, there is at least one type $t_i\in T_i$ of player $i$ who has an action $a_i\in A_i$ that satisfies
\begin{equation}
    \rho_{Z_{\beta^*}^i(t_i)}^i(L_{(a_i,\beta_{-i}^*)}^i|t_i)
    <
    \rho_{Z_{\beta^*}^i(t_i)}^i(L_{\beta^*}^i|t_i).
    \label{eq:proof:beta^* not beta^* RBNE}
\end{equation}
Consider a strategy $\hat{\beta}_i$ of this player $i$ defined by
\begin{equation*}
    \hat{\beta}_i(t_i')=
    \begin{cases}
        a_i, &\text{if } t_i'=t_i,\\
        \beta_i^*(t_i), &\text{if } t_i'\neq t_i,
    \end{cases}
\end{equation*}
which means that this player plays $a_i$ when type $t_i$ is observed and plays the strategy specified by $\beta^*$ if other types are observed.
Then, by Lemma \ref{lemma:decomposition}, we have
\begin{equation}
    \begin{aligned}
        \rho^i(L_{\hat{\beta}}^i)
        =&
        \sup_{Z^i} \mathbb{E}
        \left[ Z^i(t_i)\cdot  \rho_{Z^i(t_i)}^i(L_{\hat{\beta}}^i|t_i)\right]
        \\
        =& 
        \mathbb{E}
        \left[ Z_{\hat{\beta}}^i(t_i)\cdot  \rho_{Z_{\hat{\beta}}^i(t_i)}^i(L_{\hat{\beta}}^i|t_i)\right]
        \\
        =&
        \sum_{t_i'\neq t_i} Z_{\hat{\beta}}^i(t_i')\cdot \rho_{Z_{\hat{\beta}}^i(t_i')}^i(L_{\hat{\beta}}^i|t_i')\cdot P(t_i')
        +
        Z_{\hat{\beta}}^i(t_i)\cdot \rho_{Z_{\hat{\beta}}^i(t_i)}^i(L_{\hat{\beta}}^i|t_i)\cdot P(t_i)
        \\
        =&
        \sum_{t_i'\neq t_i} Z_{\hat{\beta}}^i(t_i')\cdot \rho_{Z_{\hat{\beta}}^i(t_i')}^i(L_{(\beta_i^*,\beta_{-i}^*)}^i|t_i')\cdot P(t_i')
        +
        Z_{\hat{\beta}}^i(t_i)\cdot \rho_{Z_{\hat{\beta}}^i(t_i)}^i(L_{(a_i,\beta_{-i}^*)}^i|t_i)\cdot P(t_i),
        \label{eq:proof:expression of rho^i(L)}
    \end{aligned}
\end{equation}
where $P(t_i)$ for all $t_i\in T_i$ and all $i\in \mathcal{I}$ denotes the probability of $t_i$ specified by the common prior $P$.
Since each ex ante risk preference $\rho^i\in\mathcal{R}$ satisfies RPRC in Definition \ref{def:complimentarity in risk revision}, condition \eqref{eq:proof:beta^* not beta^* RBNE} lead to
\begin{equation}
    \rho_{Z_{\hat{\beta}}^i(t_i)}^i(L_{(a_i,\beta_{-i}^*)}^i|t_i)
    <
    \rho_{Z_{\hat{\beta}}^i(t_i)}^i(L_{\beta^*}^i|t_i).
    \label{eq:proof:beta hat better under beta hat revision}
\end{equation}
Combining \eqref{eq:proof:beta hat better under beta hat revision} with the assumption that the common prior $P$ is fully supported, \ie, $P(t_i)>0$ for all $t_i\in T_i$ and all $i\in\mathcal{I}$ and that $Z_{\hat{\beta}}^i(t_i)>0$, we obtain from \eqref{eq:proof:expression of rho^i(L)} that
\begin{equation*}
    \begin{aligned}
        \rho^i(L_{\hat{\beta}}^i)
        &<
        \sum_{t_i'\neq t_i} Z_{\hat{\beta}}^i(t_i')\cdot \rho_{Z_{\hat{\beta}}^i(t_i')}^i(L_{\beta^*}^i|t_i')\cdot P(t_i')
        +
        Z_{\hat{\beta}}^i(t_i)\cdot \rho_{Z_{\hat{\beta}}^i(t_i)}^i(L_{\beta^*}^i|t_i)\cdot P(t_i)
        \\
        &=
        \mathbb{E}\left[ Z_{\hat{\beta}}^i(t_i)\cdot \rho_{Z_{\hat{\beta}}^i(t_i)}^i(L_{\beta^*}^i|t_i) \right]
        \\
        &\leq
        \sup_{Z^i} \mathbb{E}
        \left[ Z^i(t_i)\cdot  \rho_{Z^i(t_i)}^i(L_{\beta^*}^i|t_i)\right]
        \\
        &=\rho^i(L_{\beta^*}^i).
    \end{aligned}
\end{equation*}
Therefore, $\beta^*$ is not a RANE.
This suffices to establish the assertion in the theorem.
\end{proof}
Note that the assumption $Z_{\beta}^i(t_i)>0$ rules out situations in which certain types of players make no contribution to the ex ante risk evaluation. Specifically, this occurs when, regardless of the strategies chosen by these types of these players, the associated dual variables corresponding to their risks are identically zero.
This can be observed from Lemma \ref{lemma:decomposition} where ex ante risk is the expectation of interim risks after applying $Z_{\beta}^i(t_i)$ as a change of measure.
If, instead, only the losses corresponding to a subset of actions of a given type of a player fail to contribute to the ex ante risk evaluation, it is not sufficient to justify $Z_{\beta}^i(t_i)=0$.
Since the term $Z_{\beta}^i(t_i)$, which is applied as a change of measure, is computed by taking the conditional expectation of the ex ante dual variable $Z_\beta^i$ given type $t_i$.

We also observe that the conditions $Z_\beta^i(t_i)>0$ and $P$ is fully supported are used in a similar way in the proof of Theorem \ref{thm:RANE is RRBNE}.
However, their interpretations are inherently different. 
The assumption that $P$ is fully supported is a necessary for showing the equivalence between ex ante and interim BNE in games with incomplete information played by Bayesian players. 
If $P(t_i)=0$ for some type $t_i$, then this type has probability $0$ to be observed at the interim stage. 
In this case, this type naturally makes no contribution to the ex ante risk evaluation.
In contrast, $Z_\beta^i(t_i)=0$ only indicates that a type of a player does not contribute to the ex ante risk evaluation.
This type of player can still appear with positive probability during interim play.

\section{Effects of risk-aversion}
\label{sec:effect}
In this section, we investigate how equilibria respond to players' risk-aversion. 
In particular, we are interested in comparative statics problems where the risk preferences of players and the common prior distribution in game $\mathbb{G}$ change.
Then, we illustrate how risk-aversion provides a way to represent inconsistent belief systems in the interim formulation of games with incomplete information.
Throughout this section, we denote by $\mathcal{R}_0$ the degenerate ex ante risk preference profile which assigns an expectation risk measure to each player as her ex ante preference, \ie, $\rho^i(\cdot)=\mathbb{E}(\cdot)$ for all $\rho^i\in\mathcal{R}_0$.

\subsection{Behaviors of equilibria}
\label{sec:comparative statics}
Let $NE_{\mathbb{G}}^{ea}(\mathcal{R}, P)$ denote the set of RANE of game $\mathbb{G}$ at the ex ante stage  when the ex ante risk preference profile is $\mathcal{R}$ and the common prior is $P$.
We will also consider a generalization of game $\mathbb{G}$ by assigning subjective prior distribution $p_i\in \Delta(T) $ to player $i\in\mathcal{I}$ that may not agree with each other. 
In this scenario, we denote, with a slight abuse of notation, by $NE_{\mathbb{G}}^{ea}(\mathcal{R}, \{p_i\}_{i\in\mathcal{I}})$ the set of RANE of the game $\mathbb{G}$ at the ex ante stage generalized by the subjective priors. 
Note that, until now, our discussions have focused on the case of an agreed common prior.
It is straightforward to check that if all expectations are computed with respect to a reference probability measure and all subjective priors are absolutely continuous with respect to this reference probability measure, our equilibrium existence results can be generalized by explicitly representing the density of a subjective prior with respect to the reference measure in risk evaluations.
A fully supported reference probability measure assumed on a finite type space naturally allows this treatment.
We refer to \cite{milgrom1985distributional} for related discussions and to \cite{he2021characterization} for adopting density-weighted payoffs.

Consider a RANE strategy profile $\beta\in NE_{\mathbb{G}}^{ea}(\mathcal{R}, P)$.
According to the definition of RANE, we have for each player $i$ that $\beta_i$ is a best response to $\beta_{-i}$ with respect to the ex ante risk preference $\rho^i\in\mathcal{R}$, \ie,
\begin{equation*}
    \beta_i\in \argmin_{\beta_i'\in\Sigma_i}
    \rho^i[L^i(t,(\beta_i',\beta_{-i}))].
\end{equation*}
Recall that the dual representation of CRMs in \eqref{eq:dual representation risk measure} indicates
\begin{equation}
    \beta_i\in \argmin_{\beta_i'\in\Sigma_i} \sup_{Z^i\in \mathfrak{M}} \mathbb{E}[L^i(t,(\beta_i',\beta_{-i}))\cdot Z^i(t)],
    \label{eq:dual of argmin rho^i}
\end{equation}
where $\mathfrak{M}$ is the dual set of densities functions. 
In particular, for finite type space $T$ with $|T|=m$, we have
\begin{equation}
    \rho^i(L^i(t, \beta))= \sup_{q^i\in Q^i}\sum_{k=1}^{m} q_k^i L^i(t^k,\beta),
    \label{eq:dual of rho under finite space}
\end{equation}
where $Q^i\subset \Delta(T)$ and $\Delta(T):=\{ q\in\mathbb{R}^m: \sum_{k=1}^m q_k=1, q\geq 0 \}$.
The set $Q^i$ depends on choices of $\rho^i$ but can be assumed to be a convex and closed subset of $\Delta(T)$ as the type set $T$ is finite \cite{shapiro2013kusuoka,shapiro2021tutorial}.
Then, a probability distribution $q_*^i\in \Delta(T)$ exists for player $i$ which attains the maximum of \eqref{eq:dual of rho under finite space}.
Therefore, from \eqref{eq:dual of argmin rho^i} and \eqref{eq:dual of rho under finite space} we have for each $i\in\mathcal{I}$ that
\begin{equation}
    \beta_i \in \argmin_{\beta_i'\in\Sigma_i} \sum_{k=1}^mq_{*k}^iL^i(t^k,(\beta_i',\beta_{-i})), 
    \label{eq:risk neutral equivalent of argmin rho^i}
\end{equation}
where $q_*^i$ is a specific probability distribution for player $i$ that attains the maximum of \eqref{eq:dual of rho under finite space}.
The optimization problem \eqref{eq:risk neutral equivalent of argmin rho^i} is a risk-neutral equivalent of player $i$'s decision problem given the strategy $\beta_{-i}$ of other players if player $i$ is assumed to be risk-neutral and adopt subjective prior distribution $q_*^i$ over $T$.
Then, a strategy profile $\beta$ satisfying \eqref{eq:risk neutral equivalent of argmin rho^i} for all players simultaneously leads to $\beta\in NE_{\mathbb{G}}^{ea}(\mathcal{R}_0,\{q_*^i\}_{i\in\mathcal{I}})$.  

The above discussion can be integrated into the following result.
\begin{proposition}
For $\beta\in NE_{\mathbb{G}}^{ea}(\mathcal{R}, P)$, there is a set of risk-neutral players whose ex ante risk preferences are given by $\mathcal{R}_0$ but they have subjective prior probability distributions $q^i_*$ for $i\in \mathcal{I}$ such that $\beta\in NE_{\mathbb{G}}^{ea}(\mathcal{R}_0,\{q_*^i\}_{i\in\mathcal{I}})$.
\end{proposition}
An interim counterpart of this conclusion can also be considered. 
In this scenario, instead of focusing on the ex ante risk preferences $\rho^i$ and RANE, we turn to the usual conditional risk evaluations with $\rho^i(\cdot|t_i)$ and RABNE.
Accordingly, conditional probability $P(t_{-i}|t_i)$ is derived from the common prior $P$ and it serve as the subjective interim belief upon private observation.
Using an analogue of the arguments for the ex ante scenario in \eqref{eq:dual of rho under finite space} and \eqref{eq:risk neutral equivalent of argmin rho^i}, the conditional counterpart of the dual representation of a coherent risk functional leads to the interpretation of a RABNE using a BNE played by Bayesian players who are risk-neutral. 
Similar as in the ex ante scenario, the equivalent conditional probability measures at the interim stage that attain the optimum in the dual representation, which can be interpreted as the subjective beliefs of equivalent risk‑neutral players, may exhibit a version of disagreement that describes whether these beliefs are formed based on ex ante information.
This disagreement is commonly referred to as belief inconsistency defined as follows.
\begin{definition}
\label{def:inconsistent beliefs}
A system of interim beliefs $b=(b_1,\cdots, b_I)$ with $b_i: T_i \rightarrow  \Delta(T_{-i})$ for $i\in \mathcal{I}$ is inconsistent, if there does not exist a common prior $P \in \Delta(T)$ such that each $b_i$ is conditionally derived from $P$.    
\end{definition}
A consistent system of interim beliefs can be derived from a common prior joint probability distribution by deriving the conditional probabilities. 
The ex ante formulation of a game of incomplete information, which is the formulation that we build our model on in Section \ref{sec:model}, automatically induces this consistency in interim beliefs.
However, one can also consider the interim formulation of a game with incomplete information where interim beliefs are directly assigned to players without assuming the existence of a common prior, see, e.g., \cite{van2010interim}. 
Then, this system of interim beliefs may be inconsistent in terms of Definition \ref{def:inconsistent beliefs}.
Despite this inconsistency, interim equilibrium notions, such as BNE and RABNE, are still well-defined.

\subsection{Commonization of inconsistent beliefs}
\label{sec:commonization}
The preceding discussion indicates a connection between risk-aversion and inconsistent belief systems. While Bayesian Nash equilibrium is defined for both the ex ante and interim formulations of a game with incomplete information, belief inconsistency eliminates the role of the player referred to as “Nature”, who assigns types to other players according to the common prior probability.
Consequently, it becomes infeasible to represent the interim formulation in a corresponding extensive form \cite{maschler2020game}. 
We show below that risk-aversion provides a way to interpret inconsistent interim belief systems.

To proceed, we consider a variant of RABNE in Definition \ref{def:RABNE} where interim risk measures are directly defined with respect to  probability measures in $\Delta(T_{-i})$, which can be regarded as directly assigned subjective beliefs in an interim formulation of a game with incomplete information. 
We still refer to the corresponding interim equilibrium notion as a RABNE for simplicity.
Similar as the ex ante stage, we denote by $NE_{\mathbb{G}}^{int}(\mathcal{R}, P)$ the set of RABNE of $\mathbb{G}$ with ex ante risk preference profile $\mathcal{R}$ and common prior $P$ at the interim stage.
Here, we use this notation to indicate that the conditional risk preferences used to evaluate interim risks are the usual conditional counterparts of the ex ante risk preferences in $\mathcal{R}$ defined with respect to the conditional probabilities associated with $P$.
For a system of interim beliefs $b=(b_1,\cdots, b_I)$ where $b_i: T_i \rightarrow \Delta(T_{-i})$, we denote by $NE_{\mathbb{G}}^{int}(\mathcal{R}_0, b)$ the interim BNE played by Bayesian players whose beliefs are directed assigned to $b$.

\begin{theorem}
\label{thm:commonization}
For $\beta \in NE_{\mathbb{G}}^{int}(\mathcal{R}_0, b)$, there is a fully-supported common prior $P\in \Delta(T)$ and an ex ante risk preference profile $\Bar{\mathcal{R}}$ that assigns coherent risk measure $\Bar{\rho}^i$ to player $i\in\mathcal{I}$,  such that $\beta \in NE_{\mathbb{G}}^{int}(\Bar{\mathcal{R}}, P)$.
\end{theorem}
\begin{proof}
We first follow a procedure similar to that in Section \ref{sec:comparative statics} and then discuss the requirement on the common prior.
Given strategy profile $\beta$ and interim belief $b$, a risk-neutral type $t_i$ of player $i$ observes expected loss $\Bar{L}_{\beta}^i(t_i):=\sum_{t_{-i}\in T_{-i}}L^i((t_i,t_{-i}), \beta)b_i(t_{-i})$.
Suppose that for this type we aim to use risk measure $\Bar{\rho}^i$ to represent her preference, then the following dual representation should be satisfied 
\begin{equation}
    \Bar{\rho}^{i,t_i}(L^i(t, \beta))=\sup_{\Bar{q}^i\in \Bar{Q}^i} \sum_{t_{-i} \in T_{-i}} L^i((t_i,t_{-i}),\beta) \Bar{q}^i(t_{-i}),
    \label{eq:proof:dual of rho bar}
\end{equation}
where $\Bar{Q}^i$ is a convex and closed subset of $\Delta(T_{-i})$.
This dual representation is a necessary and sufficient condition for the risk measure $\Bar{\rho}^{i,t_i}$ being coherent.
Then, to show that a risk-neutral player with belief $b_i$ is representable using a risk-averse player with preference $\Bar{\rho}^{i,t_i}$, \ie, $\Bar{L}_{\beta}^i(t_i)=\Bar{\rho}^{i,t_i}(L^i(t,\beta))$, it suffices to show the existence of a convex and closed set $\Bar{Q}^i$ such that $b_i\in \Bar{Q}^i$ attains the maximum of the right-hand side of \eqref{eq:proof:dual of rho bar}.
The choice of the set $\Bar{Q}^i$ is not unique in general.
Observing that the right-hand side of \eqref{eq:proof:dual of rho bar} is linear in $L^i$ for fixed $\Bar{q}^i$, we can consider the set containing all convex combinations of $b_i$ and $b_i'$, where $b_i'\in \Delta(T_{-i})$ satisfies $(b_i')^T L^i < b_i^TL^i$.
Such a $b_i'$ exists except for the case where the belief $b_i$ only assigns positive probability to subsets of $T_{-i}$ on which the minimum of $L^i((t_i,t_{-i}),\beta)$ attains, \ie, $b_i=\1_{\argmax_{t_{-i}}L^i((t_i, t_{-i}),\beta)}(t_{-i})$. 
In this scenario, $\Bar{Q}^i=\{b_i\}$ still exists but is a singleton.
Since an equivalent risk measure $\Bar{\rho}^{i,t_i}$ can be identified with $\Bar{Q}^i$ for all player $i\in\mathcal{I}$ and each $\Bar{\rho}^{i,t_i}$ is defined with respect to the conditional distribution $P(t_{-i}|t_i)$, $\beta$ can be interpreted as an equilibrium outcome from a game with risk-averse players endowed with a common prior $P$.
\\
The common prior $P$ cannot be chosen arbitrarily due to the following observations.
As $\Bar{\rho}^{i,t_i}$ is defined with respect to the conditional distribution $P(t_{-i}|t_i)\in\Delta(T_{-i})$ for player $i$, dual variables that attain the maximum of the dual representation of $\Bar{\rho}^{i,t_i}$ are absolutely continuous with respect to $P(t_{-i}|t_i)$.
Then, to match the belief $b_i\in \Delta(T_{-i})$, or in other words, to make sure that $\Bar{Q}^i$ contains at least $b_i$, it is required that the support of $b_i$ is a subset of the support of $P(\cdot|t_i)$, \ie, $\text{supp}(b_i)\subseteq\text{supp}(P(\cdot|t_i))$ for all $i$.
The assumption of the common prior $P$ being fully-supported is sufficient for this construction.
\end{proof}
Note that the risk measure $\Bar{\rho}^{i,t_i}$ defined via \eqref{eq:proof:dual of rho bar} may not necessarily belong to a class of coherent risk measures which have straightforward interpretations.
Instead, it is identified with the dual set $\Bar{Q}^i$ due to the one-to-one correspondence between a coherent risk measure and its dual form.

Although further investigation is needed for selecting the constructions of $\Bar{\rho}^{i,t_i}$ and $P$, an interpretation can be draw from Theorem \ref{thm:commonization} as $b_i \in \Delta(T_{-i})$ is the only restriction imposed on the belief system $b$.
We refer to it as “commonization.” 
Specifically, Theorem \ref{thm:commonization} shows that the risk‑neutral behavioral outcome arising from strategic interaction under an inconsistent belief system $b$ can be interpreted as the behavioral outcome of risk‑averse players operating under a common prior. 
This opens an alternative avenue for studying games with incomplete information with inconsistent interim belief systems, by restoring an ex ante stage of interaction and reinstating the player “Nature.”
Games with inconsistent beliefs, in fact, constitute the majority of all games with incomplete information \cite{zamir2020bayesian}.

\section{Numerical example}
\label{sec:numerical}

\begin{figure}
    \centering
    \includegraphics[width=1\linewidth]{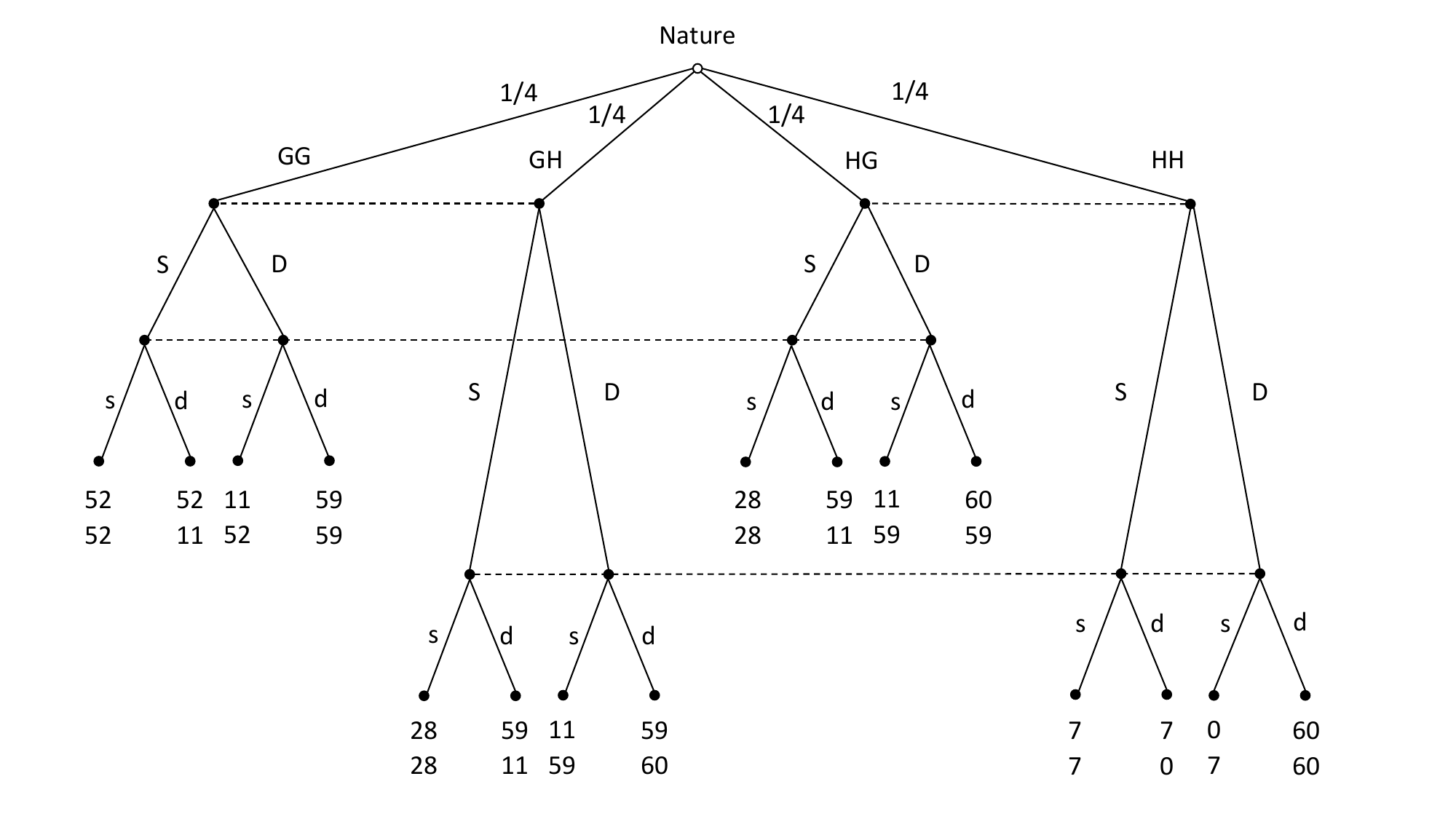}
    \caption{Game tree. Players' costs are show under the leaf nodes. For instance, under types $GG$ and strategies $(Ds)$, player $1$ observes loss $11$ and player $2$ observes loss $52$. Dashed lines represent knowledge in interim play.}
    \label{fig:gametree}
\end{figure}

In this section, we present a numerical example of a game consisting of two players, \ie,  $\mathcal{I}=\{1,2\}$, with both players having the same two available types, \ie, $T_1=T_2=\{G, H\}$. 
This is the smallest game possible for illustrating the risk-revising behaviors, as a player needs to have at least two distinct types in order to revise risk preference contingent on information and the opponent of this player needs to have at least two distinct types to introduce randomness to the decision-making environment.
There are four possible type combinations, \ie, $T=\{GG,GH,HG,HH\}$, to which the common prior $P$ assigns equal probability $1/4$.
Each player has two available actions for all types.
Player $1$ has action set $A_1=\{S,D\}$; player $2$ has action set $A_2=\{s,d\}$.
The losses received by players under different type realizations and pure strategy combinations are shown under each leaf node in the game tree in Fig. \ref{fig:gametree}.
The first number is player $1$'s loss for that case; the second number is player $2$'s loss.
For instance, under types $t_1=G$ and $t_2=G$, or simply $GG$, and pure strategy profile $(\beta_1(G)=D,\beta_2(G)=s)$, player $1$ losses $l^1=11$ and player $2$ losses $l^2=52$. 
The values of the losses shown in the game tree are chosen so that the game is symmetric.  
Additionally, we restrict attention to pure (behavioral) strategies in this example.
These options are without loss of generality and they simplify the presentation in this example.
We will also slightly abuse the notation and use $\beta_i(t_i)\beta_j(t_j)$ to denote a pair of pure strategies that may or may not belong to the same player.
For example, if player $1$ is playing pure strategy $S$ for both of her types, we simply write $SS$.
When we write $Ss$, then we are referring to a pure strategy $S$ of player $1$ and a pure strategy $s$ of player $2$ under a type combination that will be specified in the context.
Since $A_1$ and $A_2$ are distinct, it will be clear which player combination a given pair belongs to.
A pure strategy profile will be denoted $(\beta_1(G)\beta_1(H),\beta_2(G)\beta_2(H))$.
That is, $(SS,dd)$ is the pure strategy profile in which player $1$ plays $S$ whatever type she is and player $2$ plays $d$ in both of her types.

Players' ex ante risk preferences are average value-at-risk (AV@R) measures.
The AV@R measure at level $\alpha\in [0,1)$ is defined as
\begin{equation*}
    \text{AV@R}_\alpha(L):=(1-\alpha)^{-1}\int_\alpha^{1}\text{V@R}_\gamma(L)d\gamma,
\end{equation*}
where 
\begin{equation*}
    \text{V@R}_\alpha(L):=\inf\{l:F_L(l)\geq \alpha\},
\end{equation*}
and $F_L(\cdot)$ is the cumulative distribution function of $L$.
At level $\alpha=1$, define 
\begin{equation*}
    \text{AV@R}_1(L):=\esssup (L).
\end{equation*}
AV@R is an important example that belongs to the class of law-invariant coherent risk measures and it is commonly used to construct other coherent risk measures.
We refer to \cite{shapiro2013kusuoka} and the references therein for in-depth discussions on this aspect.
We adopt  $\rho^1=\rho^2=\text{AV@R}_{1/3}$ in the numerical example.

\vspace{0cm}
\begin{figure}[htb]
\centering
\includegraphics[width=0.8\linewidth]{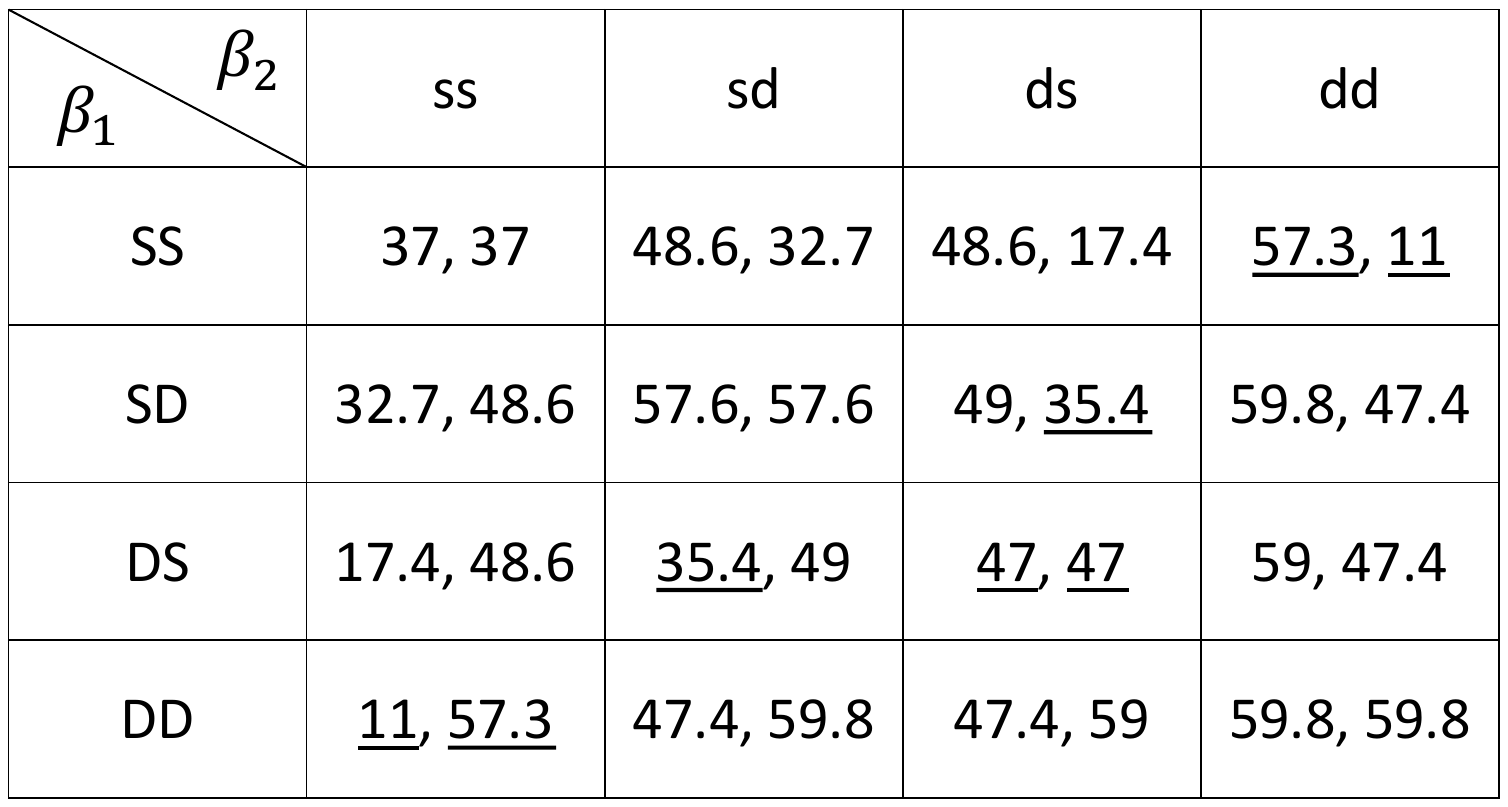}
\vspace{-1cm}
\captionof{table}[foo]{Ex ante risks. Risks corresponding to best responses are underlined.}
\label{table:ex ante risk}
\end{figure}
\vspace{0cm}
We first investigate RANE in pure strategies.
Players' ex ante risks given different pure strategy profiles are show in Table \ref{table:ex ante risk}.
For instance, the entry corresponding to $(SS,sd)$ in Table \ref{table:ex ante risk} shows that players ex ante risks are $48.6$ and $32.7$, respectively.
These risks are computed as follows.
Given strategy profile $(SS,sd)$, we observe from the game tree that player $1$'s random losses corresponding to type combinations $(GG, GH, HG, HH)$ is $L^1=[52, 59, 28, 7]$ with each possible loss having probability $1/4$.
Similarly, the random loss of player $2$ is $L^2=[52, 11, 28, 0]$ with equal probability $1/4$.
Given the random losses and their associated probabilities at the ex ante stage, we can compute $\text{AV@R}_{1/3}(L^1)=48.6$ and $\text{AV@R}_{1/3}(L^2)=32.7$.
Risks corresponding to other strategy profiles in Table \ref{table:ex ante risk} can be obtained following a similar procedure. 
Then, by locating the best responses of players based on the risks in each row and column, we observe that there are three RANE in pure strategies at the ex ante stage, namely, $(DD,ss)$, $(DS, ds)$, and $(SS,dd)$.

Moving on to the interim stage, we first focus on RABNE without considering risk revisions. 
Due to the symmetry in the game, we will investigate the strategy profiles $(DD,ss)$ and $(DS,ds)$.
The interpretations from $(DD,ss)$ and $(SS,dd)$ will be the same. 
In the interim stage, while players still adopt $\text{AV@R}_{1/3}$, all risk evaluations are performed with respect to conditional probabilities given observed private types.

The random losses and risks of the two players under strategy profile $(DD, ss)$ are shown in Tables \ref{table:DDss 1} and \ref{table:DDss 2}.
From Table \ref{table:DDss 1}, we observe that if player $2$ players $ss$, then player $1$'s best response will be $D $ for both types under $\text{AV@R}_{1/3}$ as $11<42.4$ and $8.3<22.8$.
Similarly, from Table \ref{table:DDss 2}, we observe that if player $1$ plays $DD$, then player $2$'s best response will be $s$ for both types under $\text{AV@R}_{1/3}$ due to $57.3<59$ and $46<60$.
This indicates that the strategy profile $(DD,ss)$ is also a RABNE of the game when players maintain the risk preference  $\text{AV@R}_{1/3}$.

Under strategy profile $(DS,ds)$, we only focus on player $1$ due to symmetry of the game. 
The random losses and risks in this scenario is shown in Table \ref{table:DSds 1}.
We observe that if player $2$ plays $ds$, then the best response of player $1$ will be $S$ if her type is $G$ and will be $D$ if her type is $H$.
This indicates that $(DS, ds)$ is not a RABNE at the interim stage. 

However, we show in the following that if we turn to RRBNE where players can revise their risk preferences contingent on private information, then $(DS,ds)$ can still be an interim equilibrium.
To this end, we introduce the following extended conditional versions of AV@R measures that satisfy \eqref{def:extended conditional risk functionals}:
\begin{equation}
    \text{AV@R}_{\alpha,Z_\tau}(L|\mathcal{F}_\tau):=\text{AV@R}_{1-(1-\alpha)Z_\tau}(L|\mathcal{F}_\tau).
    \label{eq:avar at random level}
\end{equation}
The conditional risk measure (\ref{eq:avar at random level}) is referred to as conditional AV@R at random level in \cite{pflug2016time} and it satisfy the following equality due to \eqref{lemma:decomposition}:
\begin{equation}
    \text{AV@R}_{\alpha}(L)=\sup \mathbb{E}[Z_\tau\cdot \text{AV@R}_{1-(1-\alpha)Z_\tau}(L|\mathcal{F}_\tau)],
    \label{eq:decomposition of avar}
\end{equation}
where the supremum is over all $\mathcal{F}_\tau$-measurable $Z_\tau$ such that $\mathbb{E}(Z_\tau)=1$, $Z_\tau\geq 0$, and $(1-\alpha)Z_\tau \leq \1$.

We assume that both players $i=1,2$ adopt $\text{AV@R}_{\alpha,Z_\tau^i}$ corresponding to the common ex ante risk preference $\text{AV@R}_\alpha$ for $\alpha=1/3$ and their private information $\mathcal{F}_\tau^i$.
Given a random loss $L^i$ for player $i$, let $Z^i$ denote the optimal dual variable at the ex ante stage corresponding to evaluation using $\rho^i(\cdot)$.
Then, $Z_\tau^i=\mathbb{E}(Z^i|\mathcal{F}_\tau^i)$ following the generic scenario discussed in Section \ref{sec:NE analysis}.
Returning to the game, player $1$ observes the random loss $L^1=[59, 11, 59, 7]$ under strategy profile $(DS, ds)$ at the ex ante stage.
The dual variable corresponding to $\text{AV@R}_{1/3}$ is $Z^1=[\frac{3}{2}, 1, \frac{3}{2}, 0]$.
Then, $Z^1(G)=\mathbb{E}[Z^1|t_1=G]=\frac{1}{2}\cdot \frac{3}{2}+\frac{1}{2}\cdot 1 =\frac{5}{4}$ and $Z^1(H)=\mathbb{E}[Z^1|t_1=H]=\frac{1}{2}\cdot \frac{3}{2}+\frac{1}{2}\cdot 0 =\frac{3}{4}$.
This leads to type $G$ of player $1$ adopting $\text{AV@R}_{1-(1-\frac{1}{3})\frac{5}{4}}=\text{AV@R}_{1/6}$ and type $H$ of player $1$ adopting $\text{AV@R}_{1-(1-\frac{1}{3})\frac{3}{4}}=\text{AV@R}_{1/2}$ at the interim stage.
Accordingly, we arrive at the risks evaluated using revised preferences in Table \ref{table:DSds 1 revise}.
In this scenario, the best response of player $1$ to $ds$ of player $2$ is $D$ under type $G$ and $S$ under type $H$.
Since the game is symmetric, the same arguments also hold for player $2$.
Therefore, $(DS,ds)$ is also a RRBNE of the game. 

In contrast, if the revised preferences are derived from the payoff corresponding to a strategy profile other than 
$(DS,ds)$ and we subsequently evaluate whether $(DS,ds) $constitutes a RRBNE, the result may be negative.
For instance, suppose that we use $(DD,sd)$.
The ex ante random loss corresponding to $(DD,sd)$ is $L^1=[11, 59, 11, 60]$.
In this case, the optimal dual variable is not unique.
Suppose we pick $Z^1=[0, \frac{3}{2}, 1, \frac{3}{2}]$.
Then, following a similar calculation as shown above we arrive at type $G$ of player $1$ adopting $\text{AV@R}_{1/2}$ and type $H$ of player $1$ adopting $\text{AV@R}_{1/6}$ at the interim stage. 
This leads to player $1$ responding to the strategy $ds$ of player $2$ using $S$ under type $G$ and using $D$ under type $H$.
Hence, $DS$ is not the best response to $ds$ for player $1$ under the risk preferences revised based on ex ante payoff corresponding to $(DD, sd)$.

\begin{figure}
    \centering
    \begin{subfigure}{0.32\textwidth}
        \centering
        \includegraphics[width=\linewidth]{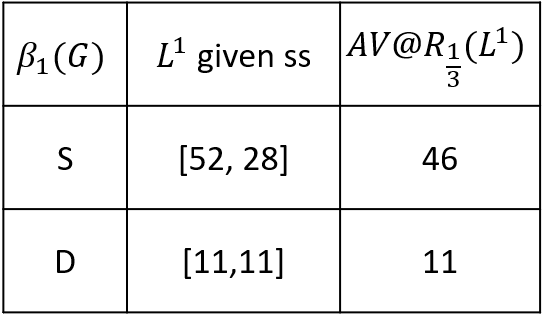}
        
        \label{fig:DDss 1G}
    \end{subfigure}
    \hspace{1.4cm}
    \begin{subfigure}{0.32\textwidth}
        \centering
        \includegraphics[width=\linewidth]{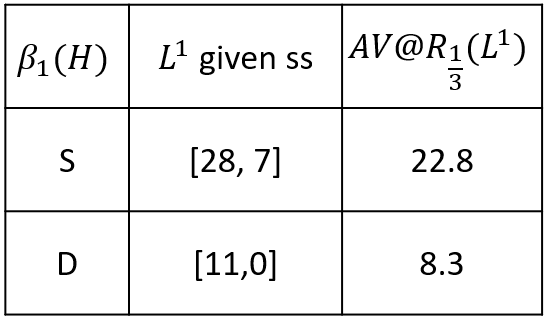}
       
        \label{fig:DDss 1H}
    \end{subfigure}
    \captionof{table}[foo]{Random losses and risks of player $1$ without risk revision under different types and pure strategies at the interim stage given player $2$ playing ss. The left table corresponds to type $G$; the right table corresponds to type $H$.}
    \label{table:DDss 1}
\end{figure}

\begin{figure}
    \centering
    \begin{subfigure}{0.4\textwidth}
        \centering
        \includegraphics[width=\linewidth]{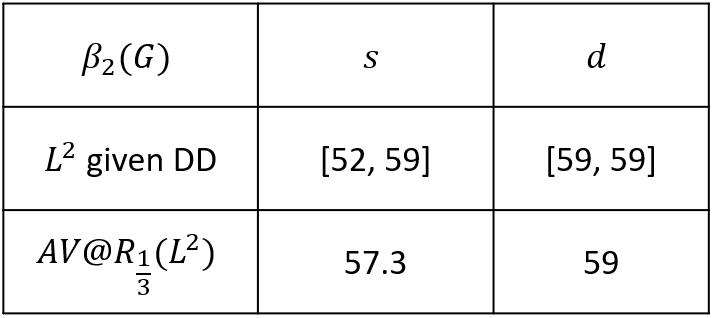}
        
        \label{fig:DDss 2G}
    \end{subfigure}
    \hspace{0.8cm}
    \begin{subfigure}{0.4\textwidth}
        \centering
        \includegraphics[width=\linewidth]{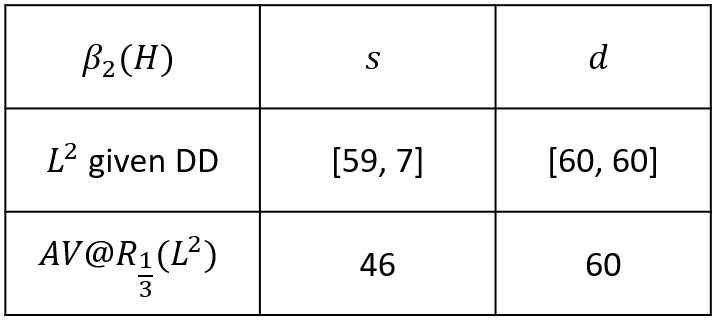}
       
        \label{fig:DDss 2H}
    \end{subfigure}
    \captionof{table}[foo]{Random losses and risks of player $2$ without risk revision under different types and pure strategies at the interim stage given player $1$ playing DD. The left table corresponds to type $G$; the right table corresponds to type $H$.}
    \label{table:DDss 2}
\end{figure}

\begin{figure}
    \centering
    \begin{subfigure}{0.32\textwidth}
        \centering
        \includegraphics[width=\linewidth]{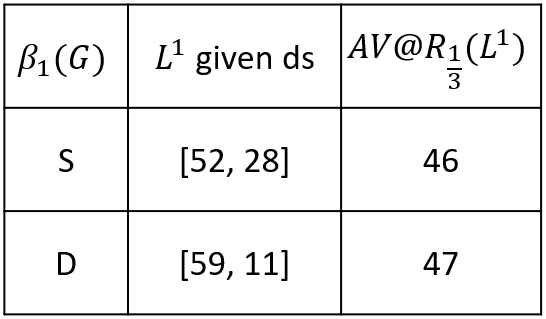}
        
        \label{fig:DDss 1G}
    \end{subfigure}
    \hspace{1.4cm}
    \begin{subfigure}{0.32\textwidth}
        \centering
        \includegraphics[width=\linewidth]{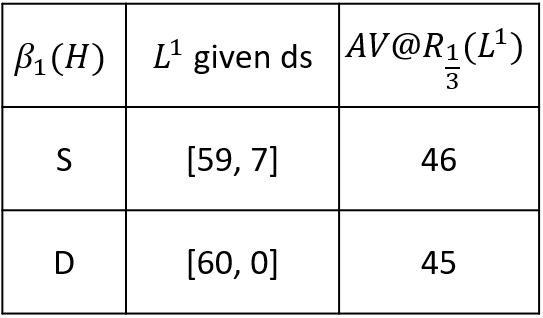}
       
        \label{fig:DDss 1H}
    \end{subfigure}
    \captionof{table}[foo]{Random losses and risks of player $1$ without risk revision under different types and pure strategies at the interim stage given player $2$ playing ds. The left table corresponds to type $G$; the right table corresponds to type $H$.}
    \label{table:DSds 1}
\end{figure}

\begin{figure}
    \centering
    \begin{subfigure}{0.32\textwidth}
        \centering
        \includegraphics[width=\linewidth]{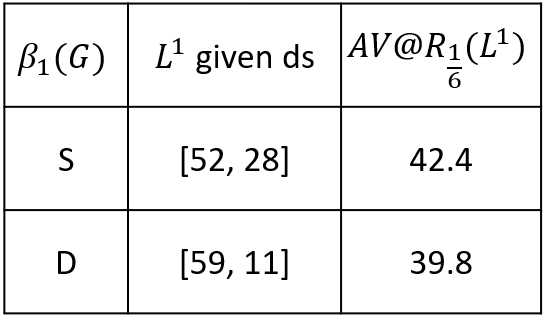}
        
        \label{fig:DDss 1G revise}
    \end{subfigure}
    \hspace{1.4cm}
    \begin{subfigure}{0.32\textwidth}
        \centering
        \includegraphics[width=\linewidth]{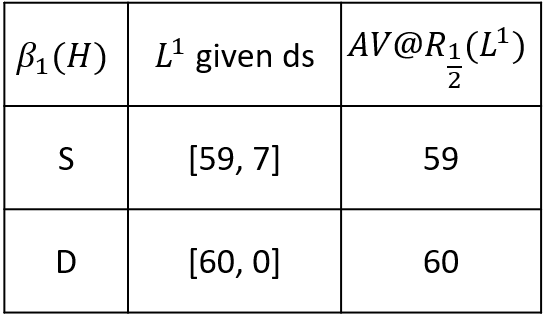}
       
        \label{fig:DDss 1H revise}
    \end{subfigure}
    \captionof{table}[foo]{Random losses and risks of player $1$ with risk revision under different types and pure strategies at the interim stage given player $2$ playing $ds$. The left table corresponds to type $G$ and revised preference $\text{AV@R}_{1/6}$; the right table corresponds to type $H$ and revised preference $\text{AV@R}_{1/2}$.}
    \label{table:DSds 1 revise}
\end{figure}

\section{Concluding remarks}
\label{sec:conclusion}
In this paper, we introduce a family of risk‑revising players in games with incomplete information. Risk‑averse equilibria at the ex ante and interim stages extend standard equilibrium notions by allowing for heterogeneous risk preferences. While these equilibria enrich behavioral outcomes in games with incomplete information, risk‑revising Bayesian Nash equilibrium bridges risk evaluations across decision stages and restricts the behaviors of risk-revising players. 
This new solution concept attaches to the ex ante defined game framework examined in this paper a proper interim behavior outcome, which lies at the core of strategic analysis in environments characterized by incomplete information.
We also discuss the potential role of risk aversion in reconciling inconsistent belief systems, particularly in settings where a common prior does not exist from which beliefs can be derived as conditional probabilities. 
We hope that the perspectives developed in this paper open up new directions for the study of a broader class of games with incomplete information and the strategic behavioral outcomes they generate.




%
%

\bibliographystyle{unsrt}  
\bibliography{references}  

\nocite{*}

\end{document}